\documentclass[english,letter paper,12pt,reqno]{article}
\usepackage{etex} 
\usepackage{array}
\usepackage{varwidth}
\usepackage{bussproofs}
\usepackage{epigraph}
\usepackage{stmaryrd}
\usepackage{mathdots}
\usepackage{amsmath, amscd, amssymb, mathrsfs, accents, amsfonts,amsthm}
\usepackage[all]{xy}
\usepackage{mathtools} 
\usepackage{tikz}
\usetikzlibrary{calc}

\AtEndDocument{\bigskip{\footnotesize%
  \textsc{Department of Mathematics, University of Melbourne} \par  
  \textit{E-mail address}: \texttt{d.murfet@unimelb.edu.au} \par
}}

\newcommand{\tagarray}{\mbox{}\refstepcounter{equation}$(\theequation)$}

\newenvironment{mathprooftree}
  {\varwidth{.9\textwidth}\centering\leavevmode}
  {\DisplayProof\endvarwidth}
  
\DeclarePairedDelimiter\ket{\lvert}{\rangle}
\DeclarePairedDelimiterX\braket[2]{\langle}{\rangle}{#1 \delimsize\vert #2}
\DeclarePairedDelimiterX\inner[2]{\langle}{\rangle}{#1,#2}

\def\drawbang{\draw[color=teal!50, line width=2pt]}
\def\drawprom{\draw[color=gray, line width=3pt]}
\def\bluenode{\node[circle,draw=blue!50,fill=blue!20]}
\def\mapnode{\node[circle,draw=black,fill=black,inner sep=0.5mm]}

\def\dernode{\node[circle,draw=black,fill=white]}
\definecolor{Myblue}{rgb}{0,0,0.6}
\usepackage[a4paper,colorlinks,citecolor=Myblue,linkcolor=Myblue,urlcolor=Myblue,pdfpagemode=None]{hyperref}

\SelectTips{cm}{}

\newtheorem{theorem}{Theorem}[section]

\newtheorem{lemma}[theorem]{Lemma}

\newtheoremstyle{example}{\topsep}{\topsep}
	{}
	{}
	{\bfseries}
	{.}
	{2pt}
	{\thmname{#1}\thmnumber{ #2}\thmnote{ #3}}
	
	\theoremstyle{example}
	\newtheorem{definition}[theorem]{Definition}
	\newtheorem{example}[theorem]{Example}
	\newtheorem{remark}[theorem]{Remark}

\def\res{\operatorname{Res}}

\def\Hom{\operatorname{Hom}}

\def\vacu{\ket{\emptyset}}

\DeclareMathOperator{\End}{End}

\DeclareMathOperator{\Spec}{Spec}

\DeclareMathOperator{\Sym}{Sym}

\DeclareMathOperator{\ILL}{ILL}
\def\inta{\bold{int}}

\begin{document}

\def\ScoreOverhang{1pt}

\def\Res{\res\!}
\newcommand{\ud}[1]{\operatorname{d}\!{#1}}
\newcommand{\Ress}[1]{\res_{#1}\!}
\newcommand{\cat}[1]{\mathcal{#1}}
\newcommand{\lto}{\longrightarrow}
\newcommand{\xlto}[1]{\stackrel{#1}\lto}
\newcommand{\mf}[1]{\mathfrak{#1}}
\newcommand{\md}[1]{\mathscr{#1}}
\newcommand{\church}[1]{\underline{#1}}
\newcommand{\prf}[1]{\underline{#1}}
\newcommand{\den}[1]{\llbracket #1 \rrbracket}
\def\l{\,|\,}
\def\sgn{\textup{sgn}}
\def\cont{\operatorname{cont}}

\title{Logic and linear algebra: an introduction}
\author{Daniel Murfet}

\maketitle

\begin{abstract} We give an introduction to logic tailored for algebraists, explaining how proofs in linear logic can be viewed as algorithms for constructing morphisms in symmetric closed monoidal categories with additional structure. This is made explicit by showing how to represent proofs in linear logic as linear maps between vector spaces. The interesting part of this vector space semantics is based on the cofree cocommutative coalgebra of Sweedler.
\end{abstract}

\setlength{\epigraphwidth}{0.6\textwidth}
\epigraph{\emph{A contrario}, for intuitionists, \emph{Modus Ponens} is not a legal advisor, it is the door open on a new world, it is the application of a function $(A \Rightarrow B)$ to an argument $(A)$ yielding a result $(B)$. Proofs are no longer those sequences of symbols created by a crazy bureaucrat, they are functions, \emph{morphisms}.}{Jean-Yves Girard, \textit{The Blind Spot}}

\tableofcontents

\section{Introduction}

Logic is familiar to mathematicians in other fields primarily as the study of the \emph{truth} of mathematical propositions, but there is an alternative tradition which views logic as being about the structure of the \emph{collection of proofs}. The contrast between these points of view can be explained by an analogy (not perfect, but perhaps helpful) between propositions in logic and differential equations. In this analogy the truth of a proposition corresponds to the \emph{existence} of a solution to a differential equation, and the set of proofs of a proposition plays the role of the space of solutions. There is often a great deal of redundancy in the space of solutions, in the sense that two solutions which are apparently different may be related by a symmetry of the underlying manifold, and thus are in essence ``the same''. Moreover there are examples, such as the Seiberg-Witten equations, where the moduli space obtained by identifying solutions related by symmetries is a compact manifold, whose ``finite'' geometry is the key to understanding the content of the original equations.

There was an analogous phenomenon at the founding of formal symbolic logic, when a central problem was to establish the consistency of arithmetic. While working on this problem Gentzen discovered a new style of presenting proofs, called the \emph{sequent calculus}, in which the set of proofs of any arithmetic proposition is revealed to contain a great deal of redundancy: many apparently different proofs actually have ``the same'' logical content. This redundancy has its origin in applications of \emph{Modus Ponens}, which is known in sequent calculus as the \emph{cut rule}. This rule creates indirection and implicitness in proofs, by making theorems depend on lemmas (whose proofs are somewhere else, hence the indirection). However this implicitness can be ``explicated'' without essentially changing the content of the proof; this deep result is known as Gentzen's \emph{Hauptsatz}. Identifying proofs related by this explicitation (called \emph{cut-elimination}) yields the much more tractable set of \emph{cut-free proofs}, the analogue of the compact manifold of solutions modulo symmetry. By reducing consistency to a problem about cut-free proofs, where it is trivial, Gentzen was able to prove the consistency of arithmetic.

\vspace{0.3cm}

The purpose of this article is to introduce the reader to this alternative tradition of logic, with its emphasis on the structure and symmetries of collections of proofs. We will do this using intuitionistic linear logic and its semantics in vector spaces and linear maps. Here the word ``semantics'' has a meaning close to what we mean by \emph{representation} in algebra: the structure of linear logic is encoded in its connectives, deduction rules, and cut-elimination transformations, and some insight into this structure can be gained by mapping it in a structure-preserving way to linear operators on vector spaces.

We begin in Section \ref{section:sketch} with some examples of proofs in linear logic, and how they can be viewed as algorithms for constructing morphisms in symmetric closed monoidal categories (omitting all details). In order to justify this we recall in Section \ref{section:lambda_calc} the formalisation of the notion of algorithm in the context of the $\lambda$-calculus, an archetypal programming language. In Section \ref{section:intro_ll} we define linear logic and revisit the examples from Section \ref{section:sketch} before turning in Section \ref{section:diagrammatics} to the semantics in vector spaces. In Section \ref{section:cut_elim} we discuss cut-elimination, and then in Section \ref{section:second} second-order linear logic. In Appendix \ref{section:appendix_cut_elim} we give a detailed example of cut-elimination and in Appendix \ref{section:example_lifting} we examine tangent maps in connection with proofs.

Most of what we have to say is well-known, with the exception of some aspects of the vector space semantics in Section \ref{section:vector_space_sem}. There are many aspects of logic and its connections with other subjects that we cannot cover: for great introductions to proof theory and the history of the subject see \cite[Chapter 1]{girard_prooftypes} and \cite[\S 1]{mellies}, and for a well-written account of analogies between logic, topology and physics see \cite{baez}.
\\

\emph{Acknowledgements.} Thanks to Nils Carqueville, Jesse Burke, Greg Restall, Shawn Standefer, Kazushige Terui and Andante.
\\

\section{A sketch of linear logic proofs as algorithms}\label{section:sketch}

Linear logic was introduced by Girard in the 1980s \cite{girard_llogic} and it has been the subject of active research ever since, in both computer science and mathematical logic. There is a close connection between linear logic and algebra, which at its root is linguistic: symmetric closed monoidal categories are ubiquitous in algebra, and their formal language is a subset of linear logic. Another way to say this is that linear logic provides a language for defining \emph{algorithms} which construct morphisms in closed symmetric monoidal categories.

For example, writing $(-) \multimap (-)$ for the internal Hom in a symmetric closed monoidal category $\cat{C}$, there is for any triple of objects $a,b,c \in \cat{C}$ a canonical map
\begin{equation}\label{eq:comp_map_intro}
(a \multimap b) \otimes (b \multimap c) \lto a \multimap c
\end{equation}
which is the internal notion of composition. It is derived from the structure of the category $\cat{C}$ in the following way: from the evaluation maps
\[
e_{a,b}: a \otimes (a \multimap b) \lto b, \qquad e_{b,c}: b \otimes (b \multimap c) \lto c
\]
and the adjunction between internal Hom and tensor we obtain a map
\begin{equation}
\xymatrix{
\Hom_{\cat{C}}( c, c ) \ar[d]^{ \Hom( e_{b,c}, 1 ) }\\
\Hom_{\cat{C}}( b \otimes (b \multimap c), c) \ar[d]^{ \Hom(e_{a,b} \otimes 1,1) }\\
\Hom_{\cat{C}}(a \otimes (a \multimap b) \otimes (b \multimap c),  c ) \ar[d]^{\cong \quad \textup{adjunction}}\\
\Hom_{\cat{C}}((a \multimap b) \otimes (b \multimap c), a \multimap c )\,.
} \label{eq:construct_comp_catt}
\end{equation}
The image of the identity on $c$ under this map is the desired composition. 

This construction is formal, in the sense that it does not depend on the nature of the particular objects $a,b,c$. The formality can be made precise by presenting the same construction using the language of linear logic:
\begin{equation}
\begin{mathprooftree}
\AxiomC{}
\UnaryInfC{$A \vdash A$}
\AxiomC{}
\UnaryInfC{$B \vdash B$}
\AxiomC{}
\UnaryInfC{$C \vdash C$}
\RightLabel{\scriptsize$\multimap L$}
\BinaryInfC{$B, B \multimap C \vdash C$}
\RightLabel{\scriptsize$\multimap L$}
\BinaryInfC{$A, A \multimap B, B \multimap C \vdash C$}
\RightLabel{\scriptsize$\multimap R$}
\UnaryInfC{$A \multimap B, B \multimap C \vdash A \multimap C$}
\end{mathprooftree}
\end{equation}
This syntactical object is called a \emph{proof}. Here $A,B,C$ are formal variables that we can think of as standing for unknown objects of $\cat{C}$ (or in fact any symmetric closed monoidal category) and $\multimap$ is a connective called linear implication. The horizontal lines stand for deduction rules. If we choose to specialise the variables of the proof to particular objects $a,b,c$ the proof acquires a shadow or interpretation in $\cat{C}$, namely, the internal composition \eqref{eq:comp_map_intro}. 

We will formally define linear logic proofs in Section \ref{section:intro_ll} and the method by which their interpretations are defined in Section \ref{section:diagrammatics}, but even without having seen any of the definitions it should seem plausible that the proof formalises the construction in \eqref{eq:construct_comp_catt}. For example the deduction rule $\multimap L$ corresponds to precomposition with an evaluation map, and $\multimap R$ to the use of adjunction. This sketch indicates how linear logic provides algorithms for the construction of canonical maps in symmetric closed monoidal categories. 
\\




Happily, symmetric closed monoidal categories are the \emph{least} interesting part of linear logic. Here is an example of an algorithm more interesting than composition: take as input an object $a$ and an endomorphism $f: a \lto a$, and return as output the square $f \circ f$. In a generic symmetric closed monoidal category there is no element in $\Hom_{\cat{C}}(a \multimap a, a \multimap a)$ which represents this algorithm internally to $\cat{C}$, in the way that we saw with composition (why?). However, this operation \emph{is} described by a proof in linear logic, namely:
\begin{equation}\label{church_2_intro}
\begin{mathprooftree}
\AxiomC{}
\UnaryInfC{$A \vdash A$}
\AxiomC{}
\UnaryInfC{$A \vdash A$}
\AxiomC{}
\UnaryInfC{$A \vdash A$}
\RightLabel{\scriptsize$\multimap L$}
\BinaryInfC{$A, A \multimap A \vdash A$}
\RightLabel{\scriptsize$\multimap L$}
\BinaryInfC{$A, A \multimap A, A \multimap A \vdash A$}
\RightLabel{\scriptsize$\multimap R$}
\UnaryInfC{$A \multimap A, A \multimap A \vdash A \multimap A$}
\RightLabel{\scriptsize der}
\UnaryInfC{$!( A \multimap A ), A \multimap A \vdash A \multimap A$}
\RightLabel{\scriptsize der}
\UnaryInfC{$!( A \multimap A), !( A \multimap A) \vdash A \multimap A$}
\RightLabel{\scriptsize ctr}
\UnaryInfC{$!( A \multimap A) \vdash A \multimap A$}
\end{mathprooftree}
\end{equation}
This proof involves a new connective $!$ called the \emph{exponential}, but we recognise that until the fourth line, this is the earlier proof with $A$ substituted for $B,C$. Since that proof was presenting the operation of composition, it only takes a small leap of imagination to see \eqref{church_2_intro} as an algorithm which takes $f: a \lto a$ and duplicates it, feeding $(f,f)$ into the composition operation to obtain $f \circ f$. Our intention in the remainder of this note is to buttress this leap of imagination with some actual mathematics, beginning with the definition of algorithms in the next section.

\section{Programs, algorithms and the $\lambda$-calculus}\label{section:lambda_calc}

One of the great achievements of 20th century mathematics was to formalise the idea of an \emph{algorithm} or \emph{program}. Around the same time as Turing defined his machines, Church gave a different formalisation of the idea using the $\lambda$-calculus \cite{church,selinger}. The two definitions are equivalent in that that they identify the same class of computable functions $\mathbb{N} \lto \mathbb{N}$, but the $\lambda$-calculus is more natural from the point of category theory, and it serves as the theoretical underpinning of functional programming languages like Lisp \cite{mccarthy} and Haskell. Intuitively, while Turing machines make precise the concept of \emph{logical state} and \emph{state transition}, the $\lambda$-calculus captures the concepts of \emph{variables} and \emph{substitution}. 

The $\lambda$-calculus is determined by its terms and a rewrite rule on those terms. The terms are to be thought of as algorithms or programs, and the rewrite rule as the method of execution of programs. A \emph{term} in the $\lambda$-calculus is either one of a countable number of variables $x,y,z,\ldots$ or an expression of the type
\begin{equation}
(M \; N) \quad \text{or} \quad (\lambda x\,.\, M)
\end{equation}
where $M,N$ are terms and $x$ is any variable. The terms of the first type are called \emph{function applications} while those of the second type are called \emph{lambda abstractions}. An example of a term, or program, that will be important throughout this note is
\begin{equation}
T := ( \lambda y \,.\, ( \lambda x \,.\, (y \,(y \; x))))\,.
\end{equation}
Note that the particular variables chosen are not important, but the pattern of occurrences of the \emph{same} variable certainly is. That is to say, we deal throughout with terms up to an equivalence relation called $\alpha$-conversion which allows us to rename variables in a consistent way. For example $T$ is equivalent to the term $( \lambda z \,.\, ( \lambda t \,.\, (z \,(z \; t))))$.

If we are supposed to think of $T$ as a program, we must describe what this program \emph{does}. The dynamic content of the $\lambda$-calculus arises from a rewrite rule called \emph{$\beta$-reduction} generated by the following basic rewrite rule
\begin{equation}\label{eq:basic_beta_reduction}
( (\lambda x \,.\, M)\, N) \longrightarrow_\beta M[N/x]
\end{equation}
where $M,N$ are terms, $x$ is a variable, and $M[N/x]$ denotes the term $M$ with all free occurrences of $x$ replaced by $N$.\footnote{There is a slight subtlety here since we may have to rename variables in order to avoid free variables in $N$ being ``captured'' as a result of this substitution, see \cite[\S 2.3]{selinger}.} We write $A \rightarrow_\beta B$ if the term $B$ is obtained from $A$ by rewriting a sub-term of $A$ according to the rule \eqref{eq:basic_beta_reduction}. The smallest reflexive, transitive and symmetric relation containing $\rightarrow_\beta$ is called \emph{$\beta$-equivalence}, written $M =_{\beta} N$ \cite[\S 2.5]{selinger}.

We think of the lambda abstraction $(\lambda x \,.\, M)$ as a program with input $x$ and body $M$, so that the $\beta$-reduction step in \eqref{eq:basic_beta_reduction} has the interpretation of our program being fed the input term $N$ which is subsequently bound to $x$ throughout $M$. A term is \emph{normal} if there are no sub-terms of the type on the left hand side of \eqref{eq:basic_beta_reduction}. In the $\lambda$-calculus computation occurs when two terms are coupled by a function application in such a way as to create a term which is not in normal form: then $\beta$-reductions are performed until a normal form (the output of the computation) is reached.\footnote{Not every $\lambda$-term may be reduced to a normal form by $\beta$-reduction because this process does not necessarily terminate. However if it does terminate then the resulting reduced term is canonically associated to the original term; this is the content of Church-Rosser theorem \cite[\S 4.2]{selinger}.}

In terms of the rewriting of terms generated by $\beta$-reduction, let us now examine what the program $T$ does when fed another term. For a term $M$, we have
\begin{equation}\label{eq:beta_reduc_dup}
(T \, M) = (( \lambda y \,.\, ( \lambda x \,.\, (y \,(y \; x))))\, M) \longrightarrow_\beta (\lambda x \, . \, (M \, (M \; x)))\,.
\end{equation}
Thus $(T \, M)$ behaves like the square of $M$, in the sense that it is a program which takes a single input $x$ and returns $(M \, (M \; x))$. For this reason $T$ is the incarnation of the number $2$ in the $\lambda$-calculus, and it is referred to as a \emph{Church numeral} \cite[\S 3.2]{selinger}.

We have now defined the $\lambda$-calculus and described our basic example of a program, the Church numeral $T$. From the descriptions we have given of their behaviour, it should not be surprising that this algorithm is closely related to the proof \eqref{church_2_intro} in the previous section, which we described there as an algorithm for taking a morphism $f: a \lto a$ in a symmetric closed monoidal category and squaring it. Next we formally define linear logic and explain in more detail the relationship between $T$ and that proof.

\section{Linear logic}\label{section:intro_ll}

\setlength{\epigraphwidth}{0.8\textwidth}
\epigraph{Linear Logic is based on the idea of resources, an idea violently negated by the contraction rule. The contraction rule states precisely that a resource is potentially infinite, which is often a sensible hypothesis, but not always. The symbol $!$ can be used precisely to distinguish those resources for which there are no limitations. From a computational point of view $!A$ means that the datum $A$ is stored in the memory and may be referenced an unlimited number of times. In some sense, $!A$ means forever.}
{J.-Y.~Girard, A.~Scedrov, P.-J.~Scott, \textit{Bounded linear logic}}


There are two main styles of formal mathematical proofs: Hilbert style systems, which most mathematicians will be exposed to as undergraduates, and natural deduction or sequent calculus systems, which are taught to students of computer science. What these two styles share is that they are about propositions (or sequents) and their proofs, which are constructed from axioms via deduction rules. The differences lie in the way that proofs are formatted and manipulated as objects on the page. In this note we will consider only the sequent calculus style of logic, since it is more naturally connected to category theory.

In \emph{intuitionistic linear logic}\footnote{We define $\ILL$ to be first-order intuitionistic linear logic without additives, and generally refer to this simply as ``linear logic'' with the exception of Section \ref{section:second} where we add quantifiers.} ($\ILL$) there are countably many propositional variables $x,y,z,\ldots$, two binary connectives $\multimap$ (linear implication), $\otimes$ (tensor) and a single unary connective $!$ (the exponential). There is a single constant $1$. The set of \emph{formulas} is defined recursively as follows: any variable or constant is a formula, and if $A,B$ are formulas then
\[
A \multimap B, \qquad A \otimes B, \qquad {!}A
\]
are formulas. An important example of a formula is $\inta_A$ defined for any formula $A$ as
\begin{equation}\label{defn:integers} 
\inta_A = {!}( A \multimap A) \multimap (A \multimap A)\,.
\end{equation}
For reasons that will become clear, $\inta_A$ is referred to the type of \emph{integers on $A$} (throughout \emph{type} is used as a synonym for formula). A \emph{sequent} is an expression of the form
\[
A_1,\ldots,A_n \vdash B
\]
with formulas $A_1,\ldots,A_n, B$ of the logic connected by a \emph{turnstile} $\vdash$. The intuitive reading of this sequent is the proposition that $B$ may be deduced from the hypotheses $A_1,\ldots,A_n$. The letters $\Gamma, \Delta$ are used to stand for arbitrary sequences of formulas, possibly empty. A \emph{proof} of a sequent is a series of deductions, beginning from tautologous \emph{axioms} of the form $A \vdash A$, which terminates with the given sequent. At each step of the proof the deduction must follow a list of \emph{deduction rules}. 

More precisely, let us define a \emph{pre-proof} to be a rooted tree whose edges are labelled with sequents. In order to follow the logical ordering, the tree is presented with its root vertex (the sequent to be proven) at the bottom of the page, and we orient edges towards the root (so downwards). The labels on incoming edges at a vertex are called \emph{hypotheses} and on the outgoing edge the \emph{conclusion}. For example consider the following tree, and its equivalent presentation in sequent calculus notation:
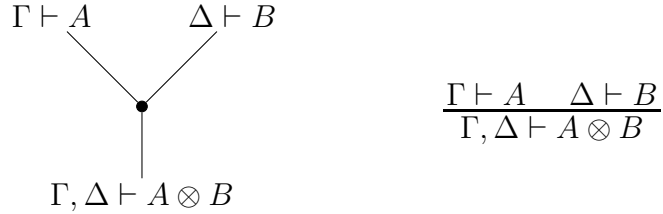
\begin{center}
\begin{tabular}{ >{\centering}m{5cm} >{\centering}m{5cm}}
\begin{tikzpicture}
[inner sep=0.5mm,scale=0.6,auto,place/.style={circle,draw=blue!50,fill=blue!20,thick},
transition/.style={circle,draw=black,fill=black}]
\node (topl) at (-2,2) {$\Gamma \vdash A$};
\node (topr) at (2,2) {$\Delta \vdash B$};
\node (bottomm) at (0,-2) {$\Gamma, \Delta \vdash A \otimes B$};
\node (o) at (0,0) [transition] {};
\draw (bottomm) -- (o);
\draw (o) -- (topr);
\draw (o) -- (topl);
\end{tikzpicture}
&
\AxiomC{$\Gamma \vdash A$} \AxiomC{$\Delta \vdash B$}
\BinaryInfC{$\Gamma, \Delta \vdash A \otimes B$}
\DisplayProof
\end{tabular}
\end{center}
Leaves are presented in the sequent calculus notation with an empty numerator.

\begin{definition} A \emph{proof} is a pre-proof together with a compatible labelling of vertices by deduction rules. The list of deduction rules is given in the first column of \eqref{deduction_rule_ax} -- \eqref{deduction_rule_right1}. A labelling is \emph{compatible} if at each vertex, the sequents labelling the incident edges match the format displayed in the deduction rule.
\end{definition}

In all deduction rules, the sets $\Gamma$ and $\Delta$ may be empty and, in particular, the promotion rule may be used with an empty premise. In the promotion rule, ${!} \Gamma$ stands for a list of formulas each of which is preceeded by an exponential modality, for example ${!}A_1,\ldots, {!}A_n$. The diagrams on the right are string diagrams and should be ignored until Section \ref{section:diagrammatics}. In particular they are \emph{not} the trees associated to proofs.

\begin{center}
\begin{tabular}{ >{\centering}m{6cm} >{\centering}m{6cm} >{\centering}m{2cm}}
\AxiomC{}
\LeftLabel{(Axiom): }
\UnaryInfC{$A \vdash A$}
\DisplayProof
&
\begin{tikzpicture}[scale=0.6,auto]
\node (top) at (0,1) {$A$};
\node (bottom) at (0,-1) {$A$};
\draw (bottom) -- (top);
\end{tikzpicture}
&
\tagarray{\label{deduction_rule_ax}}
\end{tabular}
\end{center}

\begin{center}
\begin{tabular}{ >{\centering}m{6cm} >{\centering}m{6cm} >{\centering}m{2cm}}
\AxiomC{$\Gamma, A, B, \Delta \vdash C$}
\LeftLabel{(Exchange): }
\UnaryInfC{$\Gamma, B, A, \Delta \vdash C$}
\DisplayProof
&
\begin{tikzpicture}[scale=0.6,auto]
\node (gamma) at (-2,-4) {$\Gamma$};
\node (a) at (1,-4) {$A$};
\node (b) at (-1,-4) {$B$};
\node (delta) at (2,-4) {$\Delta$};
\node (top) at (0,2) {$C$};
\draw (0,0) -- (top);
\bluenode (o) at (0,0) {};
\draw[out=90,in=180] (-2,-2) to (o);
\draw (gamma) -- (-2,-2);
\draw[out=90,in=225] (-1,-1.5) to (o);
\draw[out=90,in=270] (a) to (-1,-1.5);
\draw[out=90,in=315] (1,-1.5) to (o);
\draw[out=90,in=270] (b) to (1,-1.5);
\draw[out=90,in=0] (delta) to (o);
\end{tikzpicture}
&
\tagarray{\label{deduction_rule_ex}}
\end{tabular}
\end{center}

\begin{center}
\begin{tabular}{ >{\centering}m{6cm} >{\centering}m{6cm} >{\centering}m{2cm}}
\AxiomC{$\Gamma \vdash A$} 
\AxiomC{$\Delta', A, \Delta \vdash B$}
\LeftLabel{(Cut): }
\RightLabel{\scriptsize cut}
\BinaryInfC{$\Delta',\Gamma, \Delta \vdash B$}
\DisplayProof
&
\begin{tikzpicture}[scale=0.6,auto]
\node (top) at (0,4) {$B$};
\node (bottom_mid) at (0,-2) {$\Gamma$};
\node (bottomr) at (2,-2) {$\Delta$};
\node (bottoml) at (-2,-2) {$\Delta'$};
\bluenode (1) at (0,0) {};
\bluenode (2) at (0,2) {};
\draw (bottom_mid) -- (1);
\draw[out=90,in=270] (1) to node {$A$} (2);
\draw[out=90,in=0] (bottomr) to (2);
\draw[out=90,in=180] (bottoml) to (2);
\draw (2) -- (top);
\end{tikzpicture}
&
\tagarray{\label{deduction_rule_cut}}
\end{tabular}
\end{center}

\begin{center}
\begin{tabular}{ >{\centering}m{6cm} >{\centering}m{6cm} >{\centering}m{2cm}}
\AxiomC{$\Gamma \vdash A$} \AxiomC{$\Delta \vdash B$}
\LeftLabel{(Right $\otimes$): }
\RightLabel{\scriptsize $\otimes$-$R$}
\BinaryInfC{$\Gamma, \Delta \vdash A \otimes B$}
\DisplayProof
&
\begin{tikzpicture}[scale=0.6,auto]
\node (topl) at (-1,2) {$A$};
\node (topr) at (1,2) {$B$};
\node (bottoml) at (-1,-2) {$\Gamma$};
\node (bottomr) at (1,-2) {$\Delta$};
\bluenode (l) at (-1,0) {};
\bluenode (r) at (1,0) {};
\draw (bottoml) -- (l);
\draw (l) -- (topl);
\draw (bottomr) -- (r);
\draw (r) -- (topr);
\end{tikzpicture}
&
\tagarray{\label{deduction_rule_righttensor}}
\end{tabular}
\end{center}

\begin{center}
\begin{tabular}{ >{\centering}m{6cm} >{\centering}m{6cm} >{\centering}m{2cm}}
\AxiomC{$\Gamma, A, B, \Delta \vdash C$}
\LeftLabel{(Left $\otimes$): }
\RightLabel{\scriptsize $\otimes$-$L$}
\UnaryInfC{$\Gamma, A \otimes B, \Delta \vdash C$}
\DisplayProof
&
\begin{tikzpicture}[scale=0.6,auto]
\node (gamma) at (-2,-4) {$\Gamma$};
\node (mid) at (0,-4) {$A \otimes B$};
\node (delta) at (2,-4) {$\Delta$};
\node (top) at (0,2) {$C$};
\draw (0,0) -- (top);
\bluenode (o) at (0,0) {};
\draw[out=90,in=180] (-2,-2) to (o);
\draw (gamma) -- (-2,-2);
\draw (-1,-1.5) to (1,-1.5);
\draw[out=90,in=225] (-1,-1.5) to (o);
\draw (mid) to (0,-1.5);
\draw[out=90,in=315] (1,-1.5) to (o);
\draw[out=90,in=0] (delta) to (o);
\end{tikzpicture}
&
\tagarray{\label{deduction_rule_lefttensor}}
\end{tabular}
\end{center}

\begin{center}
\begin{tabular}{ >{\centering}m{6cm} >{\centering}m{6cm} >{\centering}m{2cm}}
\AxiomC{$A, \Gamma \vdash B$}
\LeftLabel{(Right $\multimap$): }
\RightLabel{\scriptsize $\multimap R$ }
\UnaryInfC{$\Gamma \vdash A \multimap B$}
\DisplayProof
&
\begin{tikzpicture}[scale=0.6,auto]
\node (topr) at (1,4) {$A \multimap B$};
\bluenode (o) at (1,0) {};
\node (gamma) at (2.5,-4) {$\Gamma$};
\draw[out=90,in=0] (gamma) to (o);
\draw[out=0,in=180] (-1.5,-2) to node {$A$} (o);
\draw[out=180,in=270] (-1.5,-2) to (-3,0);
\draw[out=90,in=180] (-3,0) to (1,2);
\draw (o) to node [swap] {$B$} (1,2);
\draw (1,2) to (topr);
\end{tikzpicture}
&
\tagarray{\label{deduction_rule_righthom}}
\end{tabular}
\end{center}

\begin{center}
\begin{tabular}{ >{\centering}m{6cm} >{\centering}m{6cm} >{\centering}m{2cm}}
\AxiomC{$\Gamma \vdash A$} \AxiomC{$\Delta', B, \Delta \vdash C$}
\LeftLabel{(Left $\multimap$): }
\RightLabel{\scriptsize $\multimap L$ }
\BinaryInfC{$\Delta', \Gamma, A \multimap B, \Delta \vdash C$}
\DisplayProof
&
\begin{tikzpicture}[scale=0.6,auto]
\node (topr) at (0,2) {$C$};
\bluenode (q) at (-1.5,-4) {};
\bluenode (o) at (0,0) {};
\node (delta) at (3.5,-6) {$\Delta$};
\node (deltap) at (-3.5,-6) {$\Delta'$};
\node (gamma) at (-1.5, -6) {$\Gamma$};
\node (ab) at (1.5,-6) {$A \multimap B$};
\draw[out=90,in=0] (delta) to (o);
\draw[out=90,in=180] (deltap) to (o);
\draw (o) to (topr);
\draw[out=90,in=270] (0,-2) to node [swap] {$B$} (o);
\draw[out=90,in=270] (gamma) to (q);
\draw[out=90,in=180] (q) to node {$A$} (0,-2);
\draw[out=90,in=0] (ab) to (0,-2);
\end{tikzpicture}
&
\tagarray{\label{deduction_rule_lefthom}}
\end{tabular}
\end{center}

\begin{center}
\begin{tabular}{ >{\centering}m{6cm} >{\centering}m{6cm} >{\centering}m{2cm}}
\AxiomC{$! \Gamma \vdash A$}
\LeftLabel{(Promotion): }
\RightLabel{\scriptsize prom}
\UnaryInfC{$!\Gamma \vdash !A$}
\DisplayProof
&
\begin{tikzpicture}[scale=0.6,auto]
\node (top) at (0,4) {$!A$};
\node (bottom) at (0,-4) {$!\Gamma$};
\bluenode (o) at (0,0) {};
\draw (o) to node {$A$} (0,2);
\drawbang (bottom) -- (0,-2);
\drawbang (0,2) -- (top);
\drawbang (0,-2) -- (o);
\drawprom (0,0) ellipse (2cm and 2cm);
\dernode (bottom) at (0,-2) {};
\end{tikzpicture}
&
\tagarray{\label{deduction_rule_prom}}
\end{tabular}
\end{center}

\begin{center}
\begin{tabular}{ >{\centering}m{6cm} >{\centering}m{6cm} >{\centering}m{2cm}}
\AxiomC{$\Gamma, A, \Delta \vdash B$}
\LeftLabel{(Dereliction): }
\RightLabel{\scriptsize der}
\UnaryInfC{$\Gamma, !A, \Delta \vdash B$}
\DisplayProof
&
\begin{tikzpicture}[scale=0.6,auto]
\node (top) at (0,2) {$B$};
\bluenode (o) at (0,0) {};
\dernode (d) at (0,-2) {};
\node (gamma) at (-2, -4) {$\Gamma$};
\node (delta) at (2, -4) {$\Delta$};
\node (banga) at (0,-4) {$!A$};
\drawbang (banga) -- (d);
\draw (o) -- (top);
\draw[out=90,in=270] (d) to node [swap] {$A$} (o);
\draw[out=90,in=0] (delta) to (o);
\draw[out=90,in=180] (gamma) to (o);
\end{tikzpicture}
&
\tagarray{\label{deduction_rule_der}}
\end{tabular}
\end{center}

\begin{center}
\begin{tabular}{ >{\centering}m{6cm} >{\centering}m{6cm} >{\centering}m{2cm}}
\AxiomC{$\Gamma, !A, !A, \Delta \vdash B$}
\LeftLabel{(Contraction): }
\RightLabel{\scriptsize ctr}
\UnaryInfC{$\Gamma, !A, \Delta \vdash B$}
\DisplayProof
&
\begin{tikzpicture}[scale=0.6,auto]
\node (top) at (0,2) {$B$};
\bluenode (o) at (0,0) {};
\node (gamma) at (-2, -4) {$\Gamma$};
\node (delta) at (2, -4) {$\Delta$};
\node (banga) at (0,-4) {$!A$};
\drawbang (banga) -- (0,-2);
\draw (o) -- (top);
\drawbang[out=0,in=325] (0,-2) to (o);
\drawbang[out=180,in=215] (0,-2) to (o);
\draw[out=90,in=0] (delta) to (o);
\draw[out=90,in=180] (gamma) to (o);
\end{tikzpicture}
&
\tagarray{\label{deduction_rule_contr}}
\end{tabular}
\end{center}

\begin{center}
\begin{tabular}{ >{\centering}m{6cm} >{\centering}m{6cm} >{\centering}m{2cm}}
\AxiomC{$\Gamma, \Delta \vdash B$}
\LeftLabel{(Weakening): }
\RightLabel{\scriptsize weak}
\UnaryInfC{$\Gamma, !A, \Delta \vdash B$}
\DisplayProof
&
\begin{tikzpicture}[scale=0.6,auto]
\node (top) at (0,2) {$B$};
\bluenode (o) at (0,0) {};
\node (gamma) at (-2, -3) {$\Gamma$};
\node (delta) at (2, -3) {$\Delta$};
\node (banga) at (0,-3) {$!A$};
\drawbang (banga) -- (0,-1);
\draw (o) -- (top);
\draw[out=90,in=0] (delta) to (o);
\draw[out=90,in=180] (gamma) to (o);
\node[circle,draw=teal!50,fill=teal!50,inner sep=0.5mm] at (0,-1) {};
\end{tikzpicture}
&
\tagarray{\label{deduction_rule_weak}}
\end{tabular}
\end{center}

\begin{center}
\begin{tabular}{ >{\centering}m{6cm} >{\centering}m{6cm} >{\centering}m{2cm}}
\AxiomC{$\Gamma, \Delta \vdash A$}
\LeftLabel{(Left $1$): }
\RightLabel{\scriptsize $1$-$L$}
\UnaryInfC{$\Gamma, 1, \Delta \vdash A$}
\DisplayProof
&
\begin{tikzpicture}[scale=0.6,auto]
\node (top) at (0,2) {$B$};
\bluenode (o) at (0,0) {};
\node (gamma) at (-2, -3) {$\Gamma$};
\node (delta) at (2, -3) {$\Delta$};
\node (one) at (0,-3) {$1$};
\draw (o) -- (top);
\draw (gamma) to (-2,-1.5);
\draw[out=90,in=180] (-2,-1.5) to (o);
\draw[out=90,in=0] (delta) to (o);
\draw[out=90,in=0] (one) to (-2,-1.5);
\end{tikzpicture}
&
\tagarray{\label{deduction_rule_left1}}
\end{tabular}
\end{center}

\begin{center}
\begin{tabular}{ >{\centering}m{6cm} >{\centering}m{6cm} >{\centering}m{2cm}}
\AxiomC{}
\LeftLabel{(Right $1$): }
\RightLabel{\scriptsize $1$-$R$}
\UnaryInfC{$\vdash 1$}
\DisplayProof
&
\begin{tikzpicture}[scale=0.6,auto]
\node (top) at (0,1) {$1$};
\node (bottom) at (0,-1) {$1$};
\draw (bottom) -- (top);
\end{tikzpicture}
&
\tagarray{\label{deduction_rule_right1}}
\end{tabular}
\end{center}

\begin{example}\label{example:first_occur_2} For any formula $A$ let $\church{2}_A$ denote the proof \eqref{church_2_intro} from Section \ref{section:sketch}, which we repeat here for the reader's convenience:
\begin{center}
\AxiomC{}
\UnaryInfC{$A \vdash A$}
\AxiomC{}
\UnaryInfC{$A \vdash A$}
\AxiomC{}
\UnaryInfC{$A \vdash A$}
\RightLabel{\scriptsize$\multimap L$}
\BinaryInfC{$A, A \multimap A \vdash A$}
\RightLabel{\scriptsize$\multimap L$}
\BinaryInfC{$A, A \multimap A, A \multimap A \vdash A$}
\RightLabel{\scriptsize$\multimap R$}
\UnaryInfC{$A \multimap A, A \multimap A \vdash A \multimap A$}
\RightLabel{\scriptsize der}
\UnaryInfC{$!( A \multimap A ), A \multimap A \vdash A \multimap A$}
\RightLabel{\scriptsize der}
\UnaryInfC{$!( A \multimap A), !(A \multimap A) \vdash A \multimap $}
\RightLabel{\scriptsize ctr}
\UnaryInfC{$!( A \multimap A) \vdash A \multimap A$}
\RightLabel{\scriptsize$\multimap R$}
\UnaryInfC{$\vdash \inta_A$}
\DisplayProof
\qquad
\tagarray{\label{church_2_prooftree}}
\end{center}
We also write $\church{2}_A$ for the proof of ${!}( A \multimap A ) \vdash A \multimap A$ obtained by reading the above proof up to the penultimate line. For each integer $n \ge 0$ there is a proof $\church{n}_A$ of $\inta_A$ constructed along similar lines, see \cite[\S 5.3.2]{girard_llogic} and \cite[\S 3.1]{danos}.
\end{example}

In what sense is this proof an avatar of the number $2$? In the context of the $\lambda$-calculus we appreciated the relationship between the term $T$ and the number $2$ only after we saw how $T$ interacted with other terms $M$ by forming the function application $(T \, M)$ and then finding a normal form with respect to $\beta$-equivalence. 

The analogue of function application in linear logic is the cut rule. The analogue of $\beta$-equivalence is an equivalence relation on the set of proofs of any sequent, which we write as $\rho =_{cut} \delta$ and communicate by saying that $\rho$ and $\delta$ are \emph{equivalent under cut-elimination}. Next we describe how this equivalence relation  gives a dynamic meaning to the proof $\church{2}_A$, while postponing the actual definition of cut-elimination until Section \ref{section:cut_elim}.

\begin{example}\label{example:cutagainst2} Let $\pi$ be any proof of $A \vdash A$ and consider the proof 
\begin{equation}\label{cut_against_2_early}
\begin{mathprooftree}
\AxiomC{$\pi$}
\noLine\UnaryInfC{$\vdots$}
\def\extraVskip{5pt}
\noLine\UnaryInfC{$A \vdash A$}
\RightLabel{\scriptsize $\multimap R$}
\UnaryInfC{$\vdash A \multimap A$}
\RightLabel{\scriptsize prom}
\UnaryInfC{$\vdash {!}(A \multimap A)$}
\def\extraVskip{2pt}
\AxiomC{$\church{2}_A$}
\noLine\UnaryInfC{$\vdots$}
\def\extraVskip{5pt}
\noLine\UnaryInfC{$!(A \multimap A) \vdash A \multimap A$}
\def\extraVskip{2pt}
\RightLabel{\scriptsize cut}
\BinaryInfC{$\vdash A \multimap A$}
\end{mathprooftree}
\end{equation}
We write $\delta \l \rho$ for the cut rule applied with left branch $\rho$ and right branch $\delta$, and $\operatorname{prom}( \pi )$ for the left hand branch of \eqref{cut_against_2_early}, so that the whole proof is denoted $\church{2}_A \l \operatorname{prom}( \pi )$.

Intuitively, the promotion rule ``makes perennial'' the data of the proof $\pi$ and prepares it to be used more than once (that is, in a nonlinear way). This is related to the idea of \emph{storage} in the execution of programs on a physical system. The cut rule is logically the incarnation of Modus Ponens, and from the point of view of categorical semantics is the linguistic antecedent of composition. In the context of \eqref{cut_against_2_early} it plays the role of feeding the perennial version of $\pi$ as input to $\church{2}_A$. After our experience with the $\lambda$-term $T$ it is not a surprise that the above proof is equivalent under cut-elimination to
\begin{equation}\label{cut_against_2_early_3}
\begin{mathprooftree}
\AxiomC{$\pi$}
\noLine\UnaryInfC{$\vdots$}
\def\extraVskip{5pt}
\noLine\UnaryInfC{$A \vdash A$}
\def\extraVskip{2pt}
\AxiomC{$\pi$}
\noLine\UnaryInfC{$\vdots$}
\def\extraVskip{5pt}
\noLine\UnaryInfC{$A \vdash A$}
\def\extraVskip{2pt}
\RightLabel{\scriptsize cut}
\BinaryInfC{$A \vdash A$}
\RightLabel{\scriptsize $\multimap R$}
\UnaryInfC{$\vdash A \multimap A$}
\end{mathprooftree}
\end{equation}
which is essentially the square $\pi \l \pi$.
\end{example}

With $\pi$ playing the role of the term $M$ in Section \ref{section:lambda_calc}, we see the close analogy between the program $T$ of $\lambda$-calculus and the proof $\church{2}_A$ of linear logic, with the cut \eqref{cut_against_2_early} playing the role of $(T \, M)$ and cut-elimination the role of $\beta$-reduction. A beautiful and pregnant insight, which has driven much recent progress on the border of logic and computer science, is that \emph{proofs are algorithms}: using the cut rule and cut-elimination, proofs are revealed in the guise of finite, structured objects deterministically transforming inputs to outputs.
The task of the next several sections is to explain how these algorithms can be \emph{realised} in the context of symmetric closed monoidal categories with additional structure.

\begin{remark} The analogy between proofs and programs is made precise by the \emph{Curry-Howard correspondence} and its many variants \cite[\S 6.5]{selinger} which relates programs in the simply-typed $\lambda$-calculus (and its extensions) to proofs in intuitionistic logic (and its extensions) with cut-elimination playing the role of $\beta$-reduction. We will not go into details here, but there is a chain of embeddings
\[
\big\{\textup{Simply-typed $\lambda$-calculus}\big\} \hookrightarrow \big\{\textup{Intuitionistic logic}\big\} \hookrightarrow \big\{\textup{Linear logic}\big\}
\]
using which we can translate programs in the simply-typed $\lambda$-calculus into proofs of linear logic; see \cite[\S 5.1, \S 5.3]{girard_llogic} and \cite{lafont,abramsky,benton_etal}.
\end{remark}


\section{Semantics of linear logic}\label{section:diagrammatics}

\setlength{\epigraphwidth}{0.8\textwidth}
\epigraph{Denotational semantics originated in the work of Scott and Strachey (1971) and Scott (1976) as an attempt to interpret in a non-trivial way the quotient induced on $\lambda$-terms by $\beta$-equivalence. This amounts to finding an invariant of reduction, a question which may be extended to logical systems enjoying cut-elimination. Since its introduction, denotational semantics has proved to be an absolutely essential tool in computer science and proof theory, providing a wealth of information and insights into the nature of computation and formal proofs.}{P.~Boudes, D.~Mazza, L.~Tortora de Falco, \textit{An Abstract Approach to Stratification in Linear Logic}}

The structure of the set of proofs of a sequent $\Gamma \vdash A$ modulo equivalence under cut-elimination is quite complicated, and one way to understand this structure is to model it in simpler mathematical structures. These models are called \emph{semantics}. A \emph{categorical semantics} of linear logic \cite{mellies, blue_book} assigns to each formula $A$ an object $\den{A}$ of some category and to each proof of $\Gamma \vdash A$ a morphism $\den{\Gamma} \lto \den{A}$ in such a way that two proofs equivalent under cut-elimination are assigned the same morphism; these objects and morphisms are called \emph{denotations}. The connectives of linear logic become structure on the category of denotations, and compatibility with cut-elimination imposes identities relating these structures to one another.

In its general outlines the kind of structure imposed on the category of denotations is clear from the cut-elimination transformations themselves, but there are some subtleties; for a history of the development see \cite{mellies}. The upshot is that to define a categorical semantics the category of denotations must be a closed symmetric monoidal category equipped with a comonad, which is used to model the exponential modality \cite[\S 7]{mellies}. This is a refinement of the equivalence between simply-typed $\lambda$-calculus and cartesian closed categories due to Lambek and Scott \cite{lambek}. 

In this section we present the vector space semantics of linear logic, following \cite[p.5]{hyland}. The denotations of formulas are (infinite-dimensional) vector spaces, and the denotations of proofs are linear maps. We begin by explaining the denotations of formulas in Section \ref{section:denote_formula}. The denotations of proofs is trickier, and we first review the language of string diagrams in Section \ref{section:stringdiag} before giving the semantics and examples in Section \ref{section:vector_space_sem}.

\begin{remark} The first semantics of linear logic were the coherence spaces of Girard \cite[\S 3]{girard_llogic} which are a refined form of Scott's model of the $\lambda$-calculus. Models of full linear logic with negation involve the $\star$-autonomous categories of Barr \cite{barr_auto,barr_acc,barr_autolin} and the extension to include quantifiers involves indexed monoidal categories \cite{seely}.
\end{remark}

\subsection{Denotations of formulas}\label{section:denote_formula}

Let $k$ be an algebraically closed field of characteristic zero. All vector spaces are $k$-vector spaces and $\cat{V}$ denotes the category of vector spaces (not necessarily finite-dimensional) whose tensor product $\otimes_k$ is written $\otimes$.

For a vector space $V$ let ${!} V$ denote the cofree cocommutative coalgebra generated by $V$. We will discuss the explicit form of this coalgebra in the next section; for the moment it is enough to know that it exists, is determined up to unique isomorphism, and there is a linear map $d: {!} V \lto V$ which is universal.

\begin{definition}\label{defn:denotation_objects} The \emph{denotation} $\den{A}$ of a formula $A$ is defined inductively as follows:
\begin{itemize}
\item The propositional variables $x, y, z, \ldots$ are assigned chosen finite-dimensional vector spaces $\den{x}, \den{y}, \den{z}, \ldots$;
\item $\den{1} = k$;
\item $\den{A \otimes B} = \den{A} \otimes \den{B}$;
\item $\den{A \multimap B} = \den{A} \multimap \den{B}$ which is notation for $\Hom_k(\den{A}, \den{B})$;
\item $\den{!A} = {!} \den{A}$.
\end{itemize}
The denotation of a group of formulas $\Gamma = A_1,\ldots,A_n$ is their tensor product
\[
\den{\Gamma} = \den{A_1} \otimes \cdots \otimes \den{A_n}\,.
\]
If $\Gamma$ is empty then $\den{\Gamma} = k$.
\end{definition}

\begin{example}\label{example:denotation_2}  Let $A$ be a formula whose denotation is $V = \den{A}$. Then from \eqref{defn:integers},
\[
\den{\inta_A} = \den{ {!}( A \multimap A) \multimap (A \multimap A)} = \Hom_k( {!} \End_k(V), \End_k(V) )\,.
\]
\end{example}

In order to fortify the reader for some technical material in the next section, let us give a preview of the essential features of the semantics:

\begin{example}\label{example:2_denotation_preview} Let $A$ be a formula and $V = \den{A}$. The denotation of the proof $\church{2}_A$ of $\vdash \inta_A$ from Example \ref{example:first_occur_2} will be a morphism
\begin{equation}\label{eq:2_as_comp_0}
\den{\church{2}_A}: k \lto \den{\inta_A} = \Hom_k( {!} \End_k(V), \End_k(V) )\,,
\end{equation}
or equivalently, a linear map ${!} \End_k(V) \lto \End_k(V)$. What is this linear map? It turns out (see Example \ref{example:church_2} below for details) that it is the composite
\begin{equation}\label{eq:2_as_comp}
\xymatrix@C+0.5pc{
{!} \End_k(V) \ar[r]^-{\Delta} & {!} \End_k(V) \otimes {!} \End_k(V) \ar[r]^-{d \otimes d} & \End_k(V)^{\otimes 2} \ar[r]^-{- \circ -} & \End_k(V)
}
\end{equation}
where $\Delta$ is the coproduct, $d$ is the universal map, and the last map is the composition. How to reconcile this linear map with the corresponding program in the $\lambda$-calculus, which has the meaning ``square the input function''? As we will explain below, for $\alpha \in \End_k( V )$ there is a naturally associated element $\vacu_\alpha \in {!} \End_k(V)$ with the property that
\[
\Delta \vacu_\alpha = \vacu_\alpha \otimes \vacu_\alpha, \qquad d \vacu_\alpha = \alpha\,.
\]
Then $\den{\church{2}_A}$ maps this element to
\begin{equation}\label{eq:2_deals_with_0}
\vacu_\alpha \longmapsto \vacu_\alpha \otimes \vacu_\alpha \longmapsto \alpha \otimes \alpha \longmapsto \alpha \circ \alpha\,.
\end{equation}
This demonstrates how the coalgebra ${!} \End_k(V)$ may be used to encode nonlinear maps, such as squaring an endomorphism.
\end{example}

\subsection{Background on string diagrams and coalgebras}\label{section:stringdiag}


In order to present the denotations of proofs in the vector space semantics we need to firstly present the language of string diagrams following Joyal and Street \cite{JSGoTCI,JSGoTCII,ladia,khovdia,mellies, mellies_dia} and secondly present some material on cofree coalgebras in order to give formulas for the promotion and dereliction rules.



Let $\cat{V}$ denote the category of $k$-vector spaces (not necessarily finite dimensional). Then $\cat{V}$ is symmetric monoidal and for each object $V$ the functor $V \otimes -$ has a right adjoint
\[
V \multimap - := \Hom_k(V, -)\,.
\]
In addition to the usual diagrammatics of a symmetric monoidal category, we draw the evaluation map $V \otimes (V \multimap W) \lto W$ as
\begin{equation}\label{eq:diagram_eval_map}
\begin{tikzpicture}[scale=0.4,auto,baseline=(current  bounding  box.center)]
\node (topr) at (0,3) {$W$};
\coordinate (o) at (0,0) {};
\node (bottoml) at (-3,-4) {$V$};
\node (bottomr) at (3,-4) {$V \multimap W$};
\draw[out=90,in=180] (bottoml) to (o);
\draw[out=90,in=0] (bottomr) to (o);
\draw (o) to (topr);
\end{tikzpicture}
\qquad\qquad
v \otimes \psi \mapsto \psi(v)\,.
\end{equation}
The adjoint $Y \lto X \multimap Z$ of a morphism $\phi: X \otimes Y \lto Z$ is depicted as follows:\footnote{This is somewhat against the spirit of the diagrammatic calculus, since the loop labelled $X$ is not ``real'' and is only meant as a ``picture'' to be placed at a vertex between a strand labelled $Y$ and a strand labelled $X \multimap Z$. This should not cause confusion, because we will never manipulate this strand on its own. The idea is that if $X$ were a finite-dimensional vector space, so that $X \multimap Z \cong X^{\vee} \otimes Z$, the above diagram would be absolutely valid, and we persist with the same notation even when $X$ is not dualisable. In our judgement the clarity achieved by this slight cheat justifies a little valour in the face of correctness.}
\begin{equation}\label{eq:diagram_adjoint_of_map}
\begin{tikzpicture}[scale=0.5,auto,baseline=(current  bounding  box.center)]
\node (topr) at (1,4) {$X \multimap Z$};
\mapnode (o) at (1,0) {};
\node [below] at (o.east) {$\phi$};
\node (gamma) at (2.5,-4) {$Y$};
\draw[out=90,in=0] (gamma) to (o);
\draw[out=0,in=180] (-1.5,-2) to node {$X$} (o);
\draw[out=180,in=270] (-1.5,-2) to (-3,0);
\draw[out=90,in=180] (-3,0) to (1,2);
\draw (o) to node [swap] {$Z$} (1,2);
\draw (1,2) to (topr);
\end{tikzpicture}
\qquad
y \mapsto \{ x \mapsto \phi( x \otimes y ) \}\,.
\end{equation}
Next we present the categorical construct corresponding to the exponential modality in terms of an adjunction, following Benton \cite{benton}, see also \cite[\S 7]{mellies}. Let $\cat{C}$ denote the category of counital, coassociative, cocommutative coalgebras in $\cat{V}$. In this paper whenever we say \emph{coalgebra} we mean an object of $\cat{C}$. This is a symmetric monoidal category in which the tensor product (inherited from $\cat{V}$) is cartesian, see \cite[Theorem 6.4.5]{sweedler}, \cite{barr} and \cite[\S 6.5]{mellies}. 

By results of Sweedler \cite[Chapter 6]{sweedler} the forgetful functor $L: \cat{C} \lto \cat{V}$ has a right adjoint $R$ and we set ${!} = L \circ R$, as in the following diagram:\footnote{The existence of a right adjoint to the forgetful functor can also be seen to hold more generally as a consequence of the adjoint functor theorem \cite{barr}.}
\[
\xymatrix@C+3pc{
\cat{C}\ar@<0.7ex>[r]^-{L} & \cat{V} \ar@<0.7ex>[l]^-{R}
}
\qquad
! = L \circ R\,.
\]
Both $L$ and its adjoint $R$ are monoidal functors.

For each $V$ there is a coalgebra $! V$ and a counit of adjunction $d: {!} V \lto V$. Since this map will end up being the interpretation of the dereliction rule in linear logic, we refer to it as the \emph{dereliction map}. In string diagrams it is represented by an empty circle. Although it is purely decorative, it is convenient to represent coalgebras in string diagrams drawn in $\cat{V}$ by thick lines, so that for ${!} V$ the dereliction, coproduct and counit are drawn respectively as follows:
\begin{equation}\label{eq:coalgebra_maps}
\begin{tikzpicture}[scale=0.5,auto,inner sep=1mm,baseline=(current  bounding  box.center)]
\node (top) at (0,2) {$V$};
\dernode (d) at (0,0) {};
\node (banga) at (0,-2) {$!V$};
\drawbang (banga) -- (d);
\draw (d) -- (top);
\end{tikzpicture}
\qquad\qquad\qquad
\begin{tikzpicture}[scale=0.5,auto,inner sep=1mm,baseline=(current  bounding  box.center)]
\node (topl) at (-1,2) {${!} V$};
\node (topr) at (1,2) {${!} V$};
\coordinate (d) at (0,0);
\node (banga) at (0,-2) {$!V$};
\drawbang (banga) -- (d);
\drawbang[out=0,in=270] (d) to (topr);
\drawbang[out=180,in=270] (d) to (topl);
\end{tikzpicture}
\qquad\qquad\qquad
\begin{tikzpicture}[scale=0.5,auto,inner sep=1mm,baseline=(current  bounding  box.center)]
\node (banga) at (0,-2) {$!V$};
\node[circle,draw=teal!50,fill=teal!50,inner sep=0.5mm] (blah) at (0,1) {};
\drawbang (banga) -- (blah);
\end{tikzpicture}
\end{equation}
In this paper our string diagrams involve both $\cat{V}$ and $\cat{C}$ and our convention is that white regions represent $\cat{V}$ and gray regions stand for $\cat{C}$. A standard way of representing monoidal functors between monoidal categories is using coloured regions \cite[\S 5.7]{mellies}. The image under $L$ of a morphism $\alpha: C_1 \lto C_2$ in $\cat{C}$ is drawn as a vertex in a grey region embedded into a white region. The image of a morphism $\gamma: V_1 \lto V_2$ under $R$ is drawn using a white region embedded in a gray plane. For example, the diagrams representing $L(\alpha), R(\gamma)$ and ${!} \gamma = LR(\gamma)$ are respectively
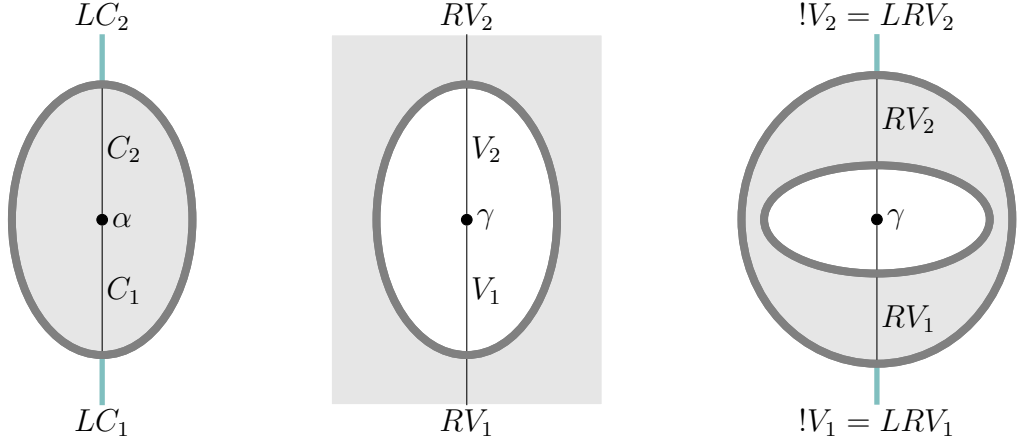
\begin{center}
\begin{tikzpicture}
[inner sep=0.5mm,scale=0.6,auto,place/.style={circle,draw=blue!50,fill=blue!20,thick},
transition/.style={circle,draw=black,fill=black}]
\node (vtop) at (0,4.5) {$L C_2$};
\node (bottom) at (0,-4.5) {$L C_1$};
\drawbang (top) to (vtop);
\coordinate (bottom_c) at (0,-3);
\drawbang (bottom) to (bottom_c);
\coordinate (top_c) at (0,3);
\draw[color=gray, line width=3pt, fill=gray!20] (0,0) ellipse (2cm and 3cm);
\node (o) at (0,0) [transition] {};
\node [right] at (o.east) {$\alpha$};
\draw (o) to node [swap] {$C_2$} (top_c);
\draw (bottom_c) to node [swap] {$C_1$} (o);
\draw[color=gray, line width=3pt] (0,0) ellipse (2cm and 3cm);
\end{tikzpicture}
\qquad \qquad
\begin{tikzpicture}
[inner sep=0.5mm,scale=0.6,auto,place/.style={circle,draw=blue!50,fill=blue!20,thick},
transition/.style={circle,draw=black,fill=black}]
\draw[color=white,fill=gray!20] (-3,-4.1) rectangle (3,4.1);
\node (vtop) at (0,4.5) {$R V_2$};
\node (bottom) at (0,-4.5) {$R V_1$};
\draw (top) to (vtop);
\coordinate (bottom_c) at (0,-3);
\draw (bottom) to (bottom_c);
\coordinate (top_c) at (0,3);
\draw[color=gray, line width=3pt, fill=white] (0,0) ellipse (2cm and 3cm);
\node (o) at (0,0) [transition] {};
\node [right] at (o.east) {$\gamma$};
\draw (o) to node [swap] {$V_2$} (top_c);
\draw (bottom_c) to node [swap] {$V_1$} (o);
\draw[color=gray, line width=3pt] (0,0) ellipse (2cm and 3cm);
\end{tikzpicture}
\qquad \qquad
\begin{tikzpicture}
[inner sep=0.5mm,scale=0.6,auto,place/.style={circle,draw=blue!50,fill=blue!20,thick},
transition/.style={circle,draw=black,fill=black}]
\node (vtop) at (0,4.5) {${!} V_2 = LR V_2$};
\node (bottom) at (0,-4.5) {${!} V_1 = LR V_1$};
\coordinate (top_c) at (0,3);
\coordinate (bottom_c) at (0,-3);
\drawbang (top) to (vtop);
\drawbang (bottom) to (bottom_c);
\draw[color=gray, line width=3pt, fill=gray!20] (0,0) ellipse (3cm and 3.2cm);
\draw[color=gray, line width=3pt, fill=white] (0,0) ellipse (2.5cm and 1.2cm);
\node (o) at (0,0) [transition] {};
\node [right] at (o.east) {$\gamma$};
\draw (o) to (0,1.2);
\draw (o) to (0,-1.2);
\draw (0,1.2) to node [swap] {$RV_2$} (0,3.2);
\draw (0,-3.2) to node [swap] {$RV_1$} (0,-1.2);
\draw[color=gray, line width=3pt] (0,0) ellipse (3cm and 3.2cm);
\draw[color=gray, line width=3pt] (0,0) ellipse (2.5cm and 1.2cm);
\end{tikzpicture}
\end{center}
The adjunction between $R$ and $L$ means that for any coalgebra $C$ and linear map $\phi: C \lto V$ there is a unique morphism of coalgebras $\Phi : C \lto {!}V$ making
\begin{equation}\label{eq:defining_philift}
\xymatrix@C+1.5pc@R+1.5pc{
C \ar[r]^-{\phi}\ar[dr]_-{\Phi} & V\\
& {!}V \ar[u]_-{d}
}
\end{equation}
commute. The lifting $\Phi$ may be constructed as the unit followed by ${!} \phi$,
\[
\Phi := \xymatrix@C+2pc{ C \ar[r] & {!} C \ar[r]^-{{!} \phi} & {!} V}
\]
and since we also use an empty circle to denote the unit $C \lto {!} C$, this has the diagrammatic representation given on the right hand side of the following diagram. The left hand side is a convenient abbreviation for this morphism, that is, for the lifting $\Phi$:
\begin{equation}\label{eq:abbrev_for_prom}
\begin{tikzpicture}[scale=0.6,auto,inner sep=1mm, baseline=(current  bounding  box.center)]
\node (top) at (0,3) {$!V$};
\node (bottom) at (0,-3) {$C$};
\mapnode (o) at (0,0) {};
\node [right] at (o.east) {$\phi$};
\draw (o) to (0,2);
\drawbang (bottom) -- (0,-1.5);
\drawbang (0,1.5) -- (top);
\drawbang (0,-1.5) -- (o);
\drawprom (0,0) ellipse (1.5cm and 1.5cm);
\dernode (bottom) at (0,-1.5) {};
\end{tikzpicture}
\qquad := \qquad
\begin{tikzpicture}[scale=0.6,auto,inner sep=1mm,baseline=(current  bounding  box.center)]
\node (vtop) at (0,4.5) {${!} V$};
\node (bottom) at (0,-5.8) {$C$};
\coordinate (top_c) at (0,3);
\coordinate (bottom_c) at (0,-3.2);
\drawbang (top_c) to (vtop);
\dernode (bottomd) at (0,-4.5) {};
\drawbang (bottomd) to node [swap] {${!} C$} (bottom_c);
\drawbang (bottom) to (bottomd);
\draw[color=gray, line width=3pt, fill=gray!20] (0,0) ellipse (3cm and 3.2cm);
\draw[color=gray, line width=3pt, fill=white] (0,0) ellipse (2.5cm and 1.2cm);
\mapnode (o) at (0,0) {};
\node [right] at (o.east) {$\phi$};
\draw (o) to (0,1.2);
\drawbang (o) to (0,-1.1);
\draw (0,1.1) to node [swap] {$RV$} (0,3.1);
\draw (0,-3.1) to node [swap] {$RC$} (0,-1.3);
\end{tikzpicture}
\end{equation}
We follow the logic literature in referring to the grey circle denoting the induced map $\Phi$ as a \emph{promotion box}. Commutativity of \eqref{eq:defining_philift} is expressed by the identity
\begin{equation}\label{eq:prom_cancel_der}
\begin{tikzpicture}[scale=0.6,auto,inner sep=1mm,baseline=(current  bounding  box.center)]
\node (top) at (0,4) {$V$};
\dernode (altop) at (0,2.5) {};
\node (bottom) at (0,-3) {$C$};
\mapnode (o) at (0,0) {};
\node [right] at (o.east) {$\phi$};
\draw (o) to (0,2);
\drawbang (bottom) -- (0,-1.5);
\drawbang (0,1.5) -- (altop);
\draw (altop) -- (top);
\drawbang (0,-1.5) -- (o);
\drawprom (0,0) ellipse (1.5cm and 1.5cm);
\dernode (bottom) at (0,-1.5) {};
\end{tikzpicture}
\qquad = \qquad
\begin{tikzpicture}[scale=0.6,auto,inner sep=1mm,baseline=(current  bounding  box.center)]
\node (top) at (0,3) {$V$};
\node (bottom) at (0,-3) {$C$};
\mapnode (o) at (0,0) {};
\node [right] at (o.east) {$\phi$};
\drawbang (bottom) -- (o);
\draw (o) -- (top);
\end{tikzpicture}
\end{equation}
which says that dereliction annihilates with promotion boxes.

The coalgebra ${!} V$ and the dereliction map ${!} V \lto V$ satisfy a universal property and are therefore unique up to isomorphism. However, to actually write down the denotations of proofs, we will need the more explicit construction which follows from the work of Sweedler \cite{sweedler} and is spelt out in \cite{murfet_coalg}. If $V$ is finite-dimensional then
\begin{equation}\label{eq:presentation_intro}
{!} V = \bigoplus_{P \in V} \Sym_P(V)
\end{equation}
where $\Sym_P(V) = \Sym(V)$ is the symmetric coalgebra. If $e_1,\ldots,e_n$ is a basis for $V$ then as a vector space $\Sym(V) \cong k[e_1,\ldots,e_n]$. The notational convention in \cite{murfet_coalg} is to denote, for elements $\nu_1,\ldots,\nu_s \in V$, the corresponding tensor in $\Sym_P(V)$ using kets
\begin{equation}\label{eq:ket_notation}
\ket{\nu_1,\ldots,\nu_s}_P := \nu_1 \otimes \cdots \otimes \nu_s \in \Sym_P(V)\,.
\end{equation}
And in particular, the identity element of $\Sym_P(V)$ is denoted by a vacuum vector
\begin{equation}\label{eq:vacuum}
\vacu_P := 1 \in \Sym_P(V)\,.
\end{equation}
We remark that if $\nu = 0$ then $\ket{\nu}_P = 0$ is the zero vector, which is distinct from $\vacu_P = 1$. To avoid unwieldy notation we sometimes write $\nu_1 \otimes \cdots \otimes \nu_s \cdot \vacu_P$ for $\ket{\nu_1,\ldots,\nu_s}_P$. With this notation the universal map $d: {!} V \lto V$ is defined by
\[
d \vacu_P = P, \quad d\ket{\nu}_P = \nu, \quad d\ket{\nu_1,\ldots,\nu_s}_P = 0 \quad s > 1\,.
\]
The coproduct on ${!} V$ is defined by
\begin{equation}\label{eq:lc_coalgebra_1}
\Delta \ket{ \nu_1, \ldots, \nu_s }_P = \sum_{I \subseteq \{ 1, \ldots, s \}} \ket{ \nu_I }_P \otimes \ket{ \nu_{I^c} }_P
\end{equation}
where $I$ ranges over all subsets including the empty set, for a subset $I = \{ i_1, \ldots, i_p \}$ we denote by $\nu_I$ the sequence $\nu_{i_1},\ldots,\nu_{i_p}$, and $I^c$ is the complement of $I$. In particular
\[
\Delta \vacu_P = \vacu_P \otimes \vacu_P\,.
\]
The counit ${!} V \lto k$ is defined by $\vacu_P \mapsto 1$ and $\ket{\nu_1,\ldots,\nu_s}_P \mapsto 0$ for $s > 0$.

When $V$ is infinite-dimensional we may define ${!} V$ as the direct limit over the coalgebras ${!} W$ for finite-dimensional subspaces $W \subseteq V$. These are sub-coalgebras of ${!} V$, and we may therefore use the same notation as in \eqref{eq:ket_notation} to denote arbitrary elements of ${!} V$. Moreover the coproduct, counit and universal map $d$ are given by the same formulas; see \cite[\S 2.1]{murfet_coalg}. A proof of the fact that the map $d: {!} V \lto V$ described above is universal among linear maps to $V$ from cocommutative coalgebras is given in \cite[Theorem 2.18]{murfet_coalg}, but as has been mentioned this is originally due to Sweedler \cite{sweedler}, see \cite[Appendix B]{murfet_coalg}.

To construct semantics of linear logic we also need an explicit description of liftings, as given by the next theorem which is \cite[Theorem 2.20]{murfet_coalg}. For a set $X$ the set of partitions of $X$ is denoted $\cat{P}_X$.

\begin{theorem}\label{theorem:describe_lifting} Let $W, V$ be finite-dimensional vector spaces and $\phi: {!}W \lto V$ a linear map. The unique lifting of $\phi$ to a morphism of coalgebras $\Phi: {!}W \lto {!}V$ is defined for vectors $P, \nu_1,\ldots,\nu_s \in W$ by
\begin{equation}\label{eq:describe_lift}
\Phi \ket{ \nu_1, \ldots, \nu_s }_P = \sum_{C \in \cat{P}_{\{1,\ldots,s\}}} \Big|\, \phi \ket{\nu_{C_{1}} }_P\,, \ldots \,, \phi \ket{ \nu_{C_{l}}}_P \Big\rangle_Q
\end{equation}
where $Q = \phi \vacu_P$ and $l$ denotes the length of the partition $C$. Moreover $\Phi \ket{\emptyset}_P = \ket{\emptyset}_Q$.
\end{theorem}


\begin{example}\label{example:lifting_trivial} The simplest example of a coalgebra is the field $k$. Any $P \in V$ determines a linear map $k \lto V$ whose lifting to a morphism of coalgebras $k \lto {!}V$ sends $1 \in k$ to the vacuum $\vacu_P$, as shown in the commutative diagram
\begin{equation}\label{eq:kaiwanxiao}
\xymatrix@C+2pc{
k \ar[r]^-{P} \ar@{.>}[dr]_-{\vacu_P} & V\\
& ! V \ar[u]_-{d}
}\,.
\end{equation}
Such liftings arise from promotions with empty premises, e.g. the proof\footnote{Incidentally, this proof helps explain why defining $\den{!A} = \Sym(\den{A})$ cannot lead to semantics of linear logic, since the denotation is a morphism of coalgebras $k \lto {!} \End_k(\den{A})$ whose composition with dereliction yields the map $k \lto \End_k( \den{A} )$ sending $1 \in k$ to the identity. But this map does not admit a lifting into the symmetric coalgebra, because it produces an infinite sum. However the symmetric coalgebra \emph{is} universal in a restricted sense and is (confusingly) sometimes also called a cofree coalgebra; see \cite[\S 4]{quillen}. For further discussion of the symmetric coalgebra in the context of linear logic see \cite{blute_fock,mellies2}.}
\begin{center}
\AxiomC{}
\UnaryInfC{$A \vdash A$}
\RightLabel{\scriptsize $\multimap R$}
\UnaryInfC{$\vdash A \multimap A$}
\RightLabel{\scriptsize prom}
\UnaryInfC{$\vdash {!}(A \multimap A)$}
\DisplayProof
\end{center}
\end{example}

\subsection{The vector space semantics}\label{section:vector_space_sem}

Recall from Definition \ref{defn:denotation_objects} the definition of $\den{A}$ for each formula $A$.

\begin{definition}\label{defn:denotation_morphism} The \emph{denotation} $\den{\pi}$ of a proof $\pi$ of $\Gamma \vdash B$ is a linear map $\den{\Gamma} \lto \den{B}$ defined by inductively assigning a string diagram to each proof tree; by the basic results of the diagrammatic calculus \cite{JSGoTCI} this diagram unambiguously denotes a linear map. The inductive construction is described by the second column in \eqref{deduction_rule_ax} -- \eqref{deduction_rule_weak}. 

In each rule we assume a morphism has already been assigned to each of the sequents in the numerator of the deduction rule. These inputs are represented by blue circles in the diagram, which computes the morphism to be assigned to the denominator. To simplify the appearance of diagrams, we adopt the convention that a strand labelled by a formula $A$ represents a strand labelled by the denotation $\den{A}$. In particular, a strand labelled with a sequence $\Gamma = A_1,\ldots,A_n$ represents a strand labelled by $\den{A_1} \otimes \cdots \otimes \den{A_n}$.
\end{definition}

Some comments:
\begin{itemize}
\item The diagram for the axiom rule \eqref{deduction_rule_ax} depicts the identity of $\den{A}$.
\item The diagram for the exchange rule \eqref{deduction_rule_ex} uses the symmetry $\den{B} \otimes \den{A} \lto \den{A} \otimes \den{B}$.
\item The diagram for the cut rule \eqref{deduction_rule_cut} depicts the composition of the two inputs.
\item The right tensor rule \eqref{deduction_rule_righttensor} depicts the tensor product of the two given morphisms, while the left tensor rule \eqref{deduction_rule_lefttensor} depicts the identity, viewed as a morphism between two strands labelled $\den{A}$ and $\den{B}$ and a single strand labelled $\den{A} \otimes \den{B}$.
\item The diagram for the right $\multimap$ rule \eqref{deduction_rule_righthom} denotes the adjoint of the input morphism as in \eqref{eq:diagram_adjoint_of_map} while the left $\multimap$ rule \eqref{deduction_rule_lefthom} uses the composition map from \eqref{eq:diagram_eval_map}.
\item The diagram for the promotion rule \eqref{deduction_rule_prom} depicts the lifting of the input to a morphism of coalgebras, as explained in \eqref{eq:abbrev_for_prom}.
\item The diagram for the dereliction rule \eqref{deduction_rule_der} depicts the composition of the input with the universal map out of the coalgebra ${!} V$. The notation for this map, and the maps in the contraction \eqref{deduction_rule_contr} and weakening rules \eqref{deduction_rule_weak} are as described in \eqref{eq:coalgebra_maps}.
\end{itemize}

\begin{remark} For this to be a valid semantics, two proofs related by cut-elimination must be assigned the same morphism. This is a consequence of the general considerations in \cite[\S 7]{mellies}. More precisely, $\cat{V}$ is a Lafont category \cite[\S 7.2]{mellies} and in the terminology of \emph{loc.cit.} the adjunction between $\cat{V}$ and $\cat{C}$ is a linear-nonlinear adjunction giving rise to a model of intuitionistic linear logic. For an explanation of how the structure of a symmetric monoidal category extrudes itself from the cut-elimination transformations, see \cite[\S 2]{mellies}.
\end{remark}


\begin{example}\label{example:church_2} We convert the proof $\church{2}_A$ of \eqref{church_2_prooftree} to a diagram in stages, beginning with the leaves. Each stage is depicted in three columns: in the first is a partial proof tree, in the second is the diagram assigned to it by Definition \ref{defn:denotation_morphism}, and in the third is the explicit linear map which is the value of the diagram. 

Recall that a strand labelled $A$ actually stands for the vector space $V = \den{A}$, so for instance the first diagram denotes a linear map $V \otimes \End_k(V) \lto V$:
\begin{center}
\begin{tabular}{>{\centering}m{6cm} >{\centering}m{5cm} >{\centering}m{4cm}}
\AxiomC{}
\UnaryInfC{$A \vdash A$}
\AxiomC{}
\UnaryInfC{$A \vdash A$}
\RightLabel{\scriptsize$\multimap L$}
\BinaryInfC{$A, A \multimap A \vdash A$}
\DisplayProof
&
\begin{tikzpicture}[scale=0.3,auto]
\node (topr) at (0,2) {$A$};
\coordinate (o) at (0,0);
\node (delta) at (2.5,-4) {$A \multimap A$};
\node (left) at (-2.5, -4) {$A$};
\draw (o) to (topr);
\draw[out=90,in=180] (left) to (o);
\draw[out=90,in=0] (delta) to (o);
\end{tikzpicture}
&
$a \otimes \alpha \mapsto \alpha(a)$
\end{tabular}
\end{center}
\begin{center}
\begin{tabular}{>{\centering}m{6cm} >{\centering}m{5cm} >{\centering}m{4cm}}
\AxiomC{}
\UnaryInfC{$A \vdash A$}
\AxiomC{}
\UnaryInfC{$A \vdash A$}
\AxiomC{}
\UnaryInfC{$A \vdash A$}
\RightLabel{\scriptsize$\multimap L$}
\BinaryInfC{$A, A \multimap A \vdash A$}
\RightLabel{\scriptsize$\multimap L$}
\BinaryInfC{$A, A \multimap A, A \multimap A \vdash A$}
\DisplayProof
&
\begin{tikzpicture}[scale=0.3,auto]
\node (topr) at (1.5,2) {$A$};
\coordinate (q) at (-3,-4);
\coordinate (o) at (1.5,0);
\node (delta) at (6,-6) {$A \multimap A$};
\node (gamma) at (-3, -6) {$A$};
\node (ab) at (1,-6) {$A \multimap A$};
\draw[out=90,in=0] (delta) to (o);
\draw (o) to (topr);
\draw[out=90,in=180] (-1,-2) to (o);
\draw[out=90,in=270] (gamma) to (q);
\draw[out=90,in=180] (q) to (-1,-2);
\draw[out=90,in=0] (ab) to (-1,-2);
\end{tikzpicture}
&
$a \otimes \alpha \otimes \beta \mapsto \beta( \alpha(a) )$
\end{tabular}
\end{center}

\begin{center}
\begin{tabular}{ >{\centering}m{6cm} >{\centering}m{5cm} >{\centering}m{4cm}}
\AxiomC{}
\UnaryInfC{$A \vdash A$}
\AxiomC{}
\UnaryInfC{$A \vdash A$}
\AxiomC{}
\UnaryInfC{$A \vdash A$}
\RightLabel{\scriptsize$\multimap L$}
\BinaryInfC{$A, A \multimap A \vdash A$}
\RightLabel{\scriptsize$\multimap L$}
\BinaryInfC{$A, A \multimap A, A \multimap A \vdash A$}
\RightLabel{\scriptsize$\multimap R$}
\UnaryInfC{$A \multimap A, A \multimap A \vdash A \multimap A$}
\DisplayProof
&
\begin{tikzpicture}[scale=0.3,auto]
\coordinate (o) at (2,0);
\node (top) at ($ (o) + (0,3) $) {$A \multimap A$}; 

\coordinate (left_meet) at ($ (o) - (3, 2) $);
\draw[out=90,in=180] (left_meet) to (o);

\node (R) at ($ (o) + (3,-5) $) {$A \multimap A$};
\node (L) at ($ (o) + (-2,-5) $) {$A \multimap A$};
\coordinate (delta) at ($ (o) - (0,5) $);
\draw[out=90,in=0] (R) to (o);

\coordinate (left_curve) at ($ (o) - (5, 4) $);
\coordinate (left_curve_mid) at ($ (o) - (6,2.5) $);
\coordinate (first_meeting_top) at ($ (o) + (0,1.5) $);
\draw[out=90,in=0] (L) to (left_meet);
\draw[out=0,in=180] (left_curve) to (left_meet);
\draw (o) to (first_meeting_top);
\draw[out=180,in=270] (left_curve) to (left_curve_mid);
\draw[out=90,in=180] (left_curve_mid) to (first_meeting_top);

\draw (first_meeting_top) to (top);
\end{tikzpicture}
&
$\alpha \otimes \beta \mapsto \beta \circ \alpha$
\end{tabular}
\end{center}

\begin{center}
\begin{tabular}{ >{\centering}m{6cm} >{\centering}m{6cm} >{\centering}m{1cm}}
\AxiomC{}
\UnaryInfC{$A \vdash A$}
\AxiomC{}
\UnaryInfC{$A \vdash A$}
\AxiomC{}
\UnaryInfC{$A \vdash A$}
\RightLabel{\scriptsize$\multimap L$}
\BinaryInfC{$A, A \multimap A \vdash A$}
\RightLabel{\scriptsize$\multimap L$}
\BinaryInfC{$A, A \multimap A, A \multimap A \vdash A$}
\RightLabel{\scriptsize$\multimap R$}
\UnaryInfC{$A \multimap A, A \multimap A \vdash A \multimap A$}
\RightLabel{\scriptsize der}
\UnaryInfC{$!( A \multimap A ), A \multimap A \vdash A \multimap A$}
\RightLabel{\scriptsize der}
\UnaryInfC{$!( A \multimap A ), !(A \multimap A) \vdash A \multimap A$}
\DisplayProof
&
\begin{tikzpicture}[scale=0.3,auto,inner sep=1mm]
\coordinate (o) at (2,0);
\node (top) at ($ (o) + (0,3) $) {}; 

\coordinate (left_meet) at ($ (o) - (3, 2) $);
\draw[out=90,in=180] (left_meet) to (o);

\dernode (R) at ($ (o) + (2,-4) $) {};
\dernode (L) at ($ (o) + (-2,-4) $) {};
\coordinate (delta) at ($ (o) - (0,5) $);
\draw[out=90,in=0] (R) to (o);
\drawbang (R) to ($ (R) - (0,2) $);
\drawbang (L) to ($ (L) - (0,2) $);

\coordinate (left_curve) at ($ (o) - (5, 4) $);
\coordinate (left_curve_mid) at ($ (o) - (6,2.5) $);
\coordinate (first_meeting_top) at ($ (o) + (0,1.5) $);
\draw[out=90,in=0] (L) to (left_meet);
\draw[out=0,in=180] (left_curve) to (left_meet);
\draw (o) to (first_meeting_top);
\draw[out=180,in=270] (left_curve) to (left_curve_mid);
\draw[out=90,in=180] (left_curve_mid) to (first_meeting_top);

\draw (first_meeting_top) to (top);
\end{tikzpicture}
&

\tagarray{\label{2_prime_pre}}

\end{tabular}
\end{center}
The map ${!} \End_k(V) \otimes {!} \End_k(V) \lto \End_k(V)$ in \eqref{2_prime_pre} is zero on $\ket{\nu_1,\ldots,\nu_s}_\alpha \otimes \ket{\mu_1,\ldots,\mu_t}_\beta$ for $\alpha,\beta \in \End_k(V)$ unless $s,t \le 1$, and in those cases it is given by
\begin{align*}
\vacu_\alpha \otimes \vacu_\beta &\longmapsto \beta \circ \alpha\\
\ket{\nu}_\alpha \otimes \vacu_\beta &\longmapsto \beta \circ \nu\\
\vacu_\alpha \otimes \ket{\mu}_\beta &\longmapsto \mu \circ \alpha\\
\ket{\nu}_\alpha \otimes \ket{\mu}_\beta &\longmapsto \mu \circ \nu\,.
\end{align*}
The next deduction rule in $\church{2}_A$ is a contraction:
\begin{center}
\begin{tabular}{ >{\centering}m{6cm} >{\centering}m{6cm} >{\centering}m{1cm}}
\AxiomC{}
\UnaryInfC{$A \vdash A$}
\AxiomC{}
\UnaryInfC{$A \vdash A$}
\AxiomC{}
\UnaryInfC{$A \vdash A$}
\RightLabel{\scriptsize$\multimap L$}
\BinaryInfC{$A, A \multimap A \vdash A$}
\RightLabel{\scriptsize$\multimap L$}
\BinaryInfC{$A, A \multimap A, A \multimap A \vdash A$}
\RightLabel{\scriptsize$\multimap R$}
\UnaryInfC{$A \multimap A, A \multimap A \vdash A \multimap A$}
\RightLabel{\scriptsize der}
\UnaryInfC{$!( A \multimap A ), A \multimap A \vdash A \multimap A$}
\RightLabel{\scriptsize der}
\UnaryInfC{$!( A \multimap A ), !(A \multimap A) \vdash A \multimap A$}
\RightLabel{\scriptsize ctr}
\UnaryInfC{$!( A \multimap A ) \vdash A \multimap A$}
\DisplayProof
&
\begin{tikzpicture}[scale=0.3,auto,inner sep=1mm]
\coordinate (o) at (2,0);
\node (top) at ($ (o) + (0,3) $) {}; 

\coordinate (left_meet) at ($ (o) - (3, 2) $);
\draw[out=90,in=180] (left_meet) to (o);

\dernode (R) at ($ (o) + (2,-4) $) {};
\dernode (L) at ($ (o) + (-2,-4) $) {};
\coordinate (delta) at ($ (o) - (0,6) $);
\draw[out=90,in=0] (R) to (o);
\drawbang[out=270,in=0] (R) to (delta);
\drawbang[out=270,in=180] (L) to (delta);
\drawbang (delta) to ($ (delta) - (0,1.5) $);

\coordinate (left_curve) at ($ (o) - (5, 4) $);
\coordinate (left_curve_mid) at ($ (o) - (6,2.5) $);
\coordinate (first_meeting_top) at ($ (o) + (0,1.5) $);
\draw[out=90,in=0] (L) to (left_meet);
\draw[out=0,in=180] (left_curve) to (left_meet);
\draw (o) to (first_meeting_top);
\draw[out=180,in=270] (left_curve) to (left_curve_mid);
\draw[out=90,in=180] (left_curve_mid) to (first_meeting_top);

\draw (first_meeting_top) to (top);
\end{tikzpicture}

&

\tagarray{\label{2_prime}}
\end{tabular}
\end{center}
Let $\phi = \den{\church{2}_A}$ be the linear map ${!} \End_k(V) \lto \End_k(V)$ denoted by this string diagram. Note that this denotation is precisely the composite \eqref{eq:2_as_comp} displayed in Example \ref{example:2_denotation_preview}. We can compute for example the image of $\ket{\nu}_\alpha$ under $\phi$ as follows:
\begin{gather*}
\ket{\nu}_\alpha \longmapsto \ket{\nu}_\alpha \otimes \vacu_\alpha + \vacu_\alpha \otimes \ket{\nu}_\alpha \longmapsto \alpha \circ \nu + \nu \circ \alpha = \{ \nu, \alpha \}\,.
\end{gather*}
This formula is perhaps not surprising as it is the tangent map of the squaring function, see Appendix \ref{section:example_lifting}. Similarly we compute that
\begin{equation}\label{eq:church_2_den}
\phi \vacu_\alpha = \alpha^2, \qquad \phi \ket{\nu}_\alpha = \{ \nu, \alpha \}\,, \qquad \phi \ket{\nu,\mu}_\alpha = \{ \nu, \mu \}\,.
\end{equation}
The final step in the proof of $\church{2}_A$ consists of moving the ${!}(A \multimap A)$ to the right side of the turnstile, which yields the final diagram:
\begin{equation}\label{church_2_diagram}
\begin{tikzpicture}[scale=0.55,auto,baseline=(current  bounding  box.center)]
\coordinate (o) at (2,0);
\node (top) at ($ (o) + (0,4) $) {$\inta_A$}; 

\coordinate (left_meet) at ($ (o) - (3, 2) $);
\draw[out=90,in=180] (left_meet) to node {$A$} (o);

\dernode (R) at ($ (o) + (2,-3) $) {};
\dernode (L) at ($ (o) + (-2,-3) $) {};
\coordinate (delta) at ($ (o) - (0,5) $);
\draw[out=90,in=0] (R) to node [swap] {$A \multimap A$} (o);
\drawbang[out=0,in=270] (delta) to (R);
\drawbang[out=180,in=270] (delta) to (L);

\coordinate (left_curve) at ($ (o) - (5, 4) $);
\coordinate (left_curve_mid) at ($ (o) - (6,2.5) $);
\coordinate (first_meeting_top) at ($ (o) + (0,1.5) $);
\draw[out=90,in=0] (L) to node [swap] {$A \multimap A$} (left_meet);
\draw[out=0,in=180] (left_curve) to node {$A$} (left_meet);
\draw (o) to node [swap] {$A$} (first_meeting_top);
\draw[out=180,in=270] (left_curve) to (left_curve_mid);
\draw[out=90,in=180] (left_curve_mid) to (first_meeting_top);

\coordinate (second_meeting_top) at ($ (first_meeting_top) + (0,1.5) $);
\draw (first_meeting_top) to node [swap] {$A \multimap A$} (second_meeting_top);
\draw (second_meeting_top) to (top);

\coordinate (curve_bottom) at ($ (left_meet) - (0,5) $);
\coordinate (curve_left) at ($ (o) - (7.5, 2) $);
\drawbang[out=0,in=270] (curve_bottom) to (delta);
\drawbang[out=180,in=270] (curve_bottom) to (curve_left);
\drawbang[out=90,in=180] (curve_left) to node {$!(A \multimap A)$} (second_meeting_top);
\end{tikzpicture}
\end{equation}
The denotation of this diagram is the same morphism $\phi$. The reader might like to compare this style of diagram for $\church{2}_A$ to the corresponding proof-net in \cite[\S 5.3.2]{girard_llogic}. 
\end{example}

In Example \ref{example:2_denotation_preview} we sketched how to recover the function $\alpha \mapsto \alpha^2$ from the denotation of the Church numeral $\church{2}_A$, but now we can put this on a firmer footing. There is \emph{a priori} no linear map $\den{B} \lto \den{C}$ associated to a proof $\pi$ of ${!} B \vdash C$ but there is a function $\den{\pi}_{nl}$ ($nl$ standing for \emph{nonlinear}) defined on $P \in \den{B}$ by lifting to $! \den{B}$ and then applying $\den{\pi}$:
\begin{equation}\label{eq:lifting_vacua}
\xymatrix@C+2pc{
k \ar[r]^-{P} \ar@{.>}[dr]_-{\vacu_P} & \den{B} & \den{C}\\
& ! \den{B} \ar[u]_-{d} \ar[ur]_-{\den{\pi}}
}\,.
\end{equation}
That is,

\begin{definition}\label{defn:nonlinear_denotation}
The function $\den{\pi}_{nl}: \den{B} \lto \den{C}$ is defined by $\den{\pi}_{nl}(P) = \den{\pi} \vacu_P$.
\end{definition}

The discussion above shows that, with $V = \den{A}$,

\begin{lemma}\label{lemma:nonlinear_recover} $\den{\church{2}_A}_{nl}: \End_k(V) \lto \End_k(V)$ is the map $\alpha \mapsto \alpha^2$.
\end{lemma}

This completes our explanation of how to represent the Church numeral $\church{2}_A$ as a linear map. From this presentation we see clearly that the non-linearity of the map $\alpha \mapsto \alpha^2$ is concentrated in the promotion step \eqref{eq:lifting_vacua} where the input vector $\alpha$ is turned into a vacuum $\vacu_\alpha \in {!} \End_k(V)$. The promotion step is non-linear since $\vacu_\alpha + \vacu_\beta$ is not a morphism of coalgebras and thus cannot be equal to $\vacu_{\alpha + \beta}$. After this step, the duplication, dereliction and composition shown in \eqref{eq:2_deals_with_0} are all linear. 


\begin{example}\label{example:2_promotion} Applied to $\church{2}_A$ the promotion rule generates a new proof
\begin{prooftree}
\AxiomC{$\church{2}_A$}
\noLine\UnaryInfC{$\vdots$}
\def\extraVskip{5pt}
\noLine\UnaryInfC{$!(A \multimap A) \vdash A \multimap A$}
\def\extraVskip{2pt}
\RightLabel{\scriptsize prom}
\UnaryInfC{$!( A \multimap A) \vdash {!}(A \multimap A)$}
\end{prooftree}
which we denote $\operatorname{prom}( \church{2}_A )$. By definition the denotation $\Phi := \den{ \operatorname{prom}( \church{2}_A ) }$ of this proof is the unique morphism of coalgebras
\[
\Phi: {!} \End_k(V) \lto {!} \End_k(V)
\]
with the property that $d \circ \Phi = \phi$, where $\phi = \den{\church{2}_A}$ is as in \eqref{eq:2_as_comp}. From \eqref{eq:church_2_den} and Theorem \ref{theorem:describe_lifting} we compute that, for example
\begin{align*}
\Phi \vacu_\alpha &= \vacu_{\alpha^2}\,,\\
\Phi \ket{\nu}_\alpha &= \{ \nu, \alpha \} \cdot \vacu_{\alpha^2},\\
\Phi \ket{\nu, \mu}_\alpha &= \big( \{ \nu, \mu \} + \{ \nu, \alpha \} \otimes \{ \mu, \alpha \} \big) \cdot \vacu_{\alpha^2}\,,\\
\Phi \ket{\nu, \mu, \theta}_\alpha &= \big( \{ \nu, \mu \} \otimes \{ \theta, \alpha \} + \{ \theta, \mu \} \otimes \{ \nu, \alpha \} + \{ \nu, \theta \} \otimes \{ \mu, \alpha \}\\
&+ \{ \nu, \alpha \} \otimes \{ \mu, \alpha \} \otimes \{ \theta, \alpha \} \big) \cdot \vacu_{\alpha^2}\,.
\end{align*}
Note that the commutators, e.g. $\{ \nu, \alpha \}$ are defined using the product internal to $\End_k(V)$, whereas inside the bracket in the last two lines, the terms $\{ \nu, \alpha \}$ and $\{ \mu, \alpha \}$ are multiplied in the algebra $\Sym( \End_k(V) )$ before being made to act on the vacuum.
\end{example}

\begin{remark}\label{remark_otherpromotion} As we have already mentioned, there are numerous semantics of linear logic defined using topological vector spaces. The reader curious about how these vector space semantics are related to the ``relational'' style of semantics such as coherence spaces should consult Ehrhard's paper on finiteness spaces \cite{ehrhard}.

One can generate simple variants on the vector space semantics by looking at comonads on the category $\cat{V}$ defined by truncations on the coalgebras ${!} V$. For example, given a vector space $V$, let ${!}_0 V$ denote the subspace of ${!} V$ generated by the vacua $\vacu_P$. This is the free space on the underlying \emph{set} of $V$, and it is a subcoalgebra of ${!} V$ given as a coproduct of trivial coalgebras. It is easy to see that this defines an appropriate comonad and gives rise to a semantics of linear logic \cite[\S 4.3]{valiron}. However the semantics defined using ${!} V$ is more interesting, because universality allows us to mix in arbitrary coalgebras $C$. In Appendix \ref{section:example_lifting} we examine the simplest example where $C$ is the dual of the algebra $k[t]/t^2$ and relate this to tangent vectors.
\end{remark}

\section{Cut-elimination}\label{section:cut_elim}

The dynamical part of linear logic is a rewrite rule on proofs called \emph{cut-elimination}, which defines the way that proofs interact and is the analogue in linear logic of $\beta$-reduction in the $\lambda$-calculus. Collectively these interactions give the connectives their meaning. The full set of rewrite rules is given in \cite[Section 3]{mellies} and in the alternative language of proof-nets in \cite[\S 4]{girard_llogic}, \cite[p.18]{pagani}. We will not reproduce the full list here because it takes up several pages and there are some subtle cases. That being said, we do not wish to leave the reader with the impression that these rules are mysterious: the key rules have a clear conceptual content, which we elucidate in this section by presenting a worked example involving our favourite proof $\church{2}_A$ and the string diagram transformations associated to the rewrites.

\begin{definition} Let $\Gamma \vdash A$ be a sequent of linear logic and $\cat{P}$ the set of proofs. There is a binary relation $\rightsquigarrow$ on $\cat{P}$ defined by saying that $\pi \rightsquigarrow \rho$ if $\rho$ is an immediate reduction of $\pi$ by one of the rewrite rules listed in \cite[Section 3]{mellies}. These rewrite rules are sometimes referred to as \emph{cut-elimination transformations}.

We write $=_{cut}$ for the smallest reflexive, transitive and symmetric relation on $\cat{P}$ containing $\rightsquigarrow$ and call this \emph{equivalence under cut-elimination}.
\end{definition}


\begin{example}
In the context of Example \ref{example:cutagainst2} there is a sequence of transformations
\[
\church{2}_A \l \operatorname{prom}( \pi ) = \pi_0 \rightsquigarrow \pi_1 \rightsquigarrow \cdots \rightsquigarrow \pi \l \pi
\]
which begins with the rewrites \cite[\S 3.9.3]{mellies} and \cite[\S 3.9.1]{mellies}.
\end{example}

\begin{example}\label{example:cut_elim_examples} The cut-elimination transformation \cite[\S 3.8.2]{mellies} tells us that
\begin{center}
\begin{tabular}{ >{\centering}m{7cm} >{\centering}m{4cm} >{\centering}m{3cm}}
\AxiomC{$\pi_2$}
\noLine\UnaryInfC{$\vdots$}
\def\extraVskip{5pt}
\noLine\UnaryInfC{$A \vdash B$}
\def\extraVskip{2pt}
\RightLabel{\scriptsize $\multimap R$}
\UnaryInfC{$\vdash A \multimap B$}
\AxiomC{$\pi_1$}
\noLine\UnaryInfC{$\vdots$}
\def\extraVskip{5pt}
\noLine\UnaryInfC{$\Gamma \vdash A$}
\def\extraVskip{2pt}
\AxiomC{$\pi_3$}
\noLine\UnaryInfC{$\vdots$}
\def\extraVskip{5pt}
\noLine\UnaryInfC{$B \vdash C$}
\def\extraVskip{2pt}
\RightLabel{\scriptsize $\multimap L$}
\BinaryInfC{$\Gamma, A \multimap B \vdash C$}
\RightLabel{\scriptsize cut}
\BinaryInfC{$\Gamma \vdash C$}
\DisplayProof
&
\begin{tikzpicture}[scale=0.35,auto]
\node (top) at (0,5) {$C$};
\mapnode (mid_top) at (0,2) {};
\node [above] at (mid_top.east) {$\;\;\;\;\;\;\den{\pi_3}$};
\node (bottoml) at (-3,-5) {$\Gamma$};
\mapnode (pi1) at (-3,-3) {};
\node [above] at (pi1.east) {$\;\;\;\;\;\;\;\;\;\den{\pi_1}$};
\coordinate (o) at (0,0);
\draw (o) -- (mid_top);
\draw (mid_top) -- (top);
\draw (bottoml) -- (pi1);
\draw[out=90,in=180] (pi1) to (o);
\mapnode (other) at (3,-3) {};
\coordinate (other_up) at (3,-2) {};
\node [right] at (other.east) {$\den{\pi_2}$};
\draw[out=90,in=0] (other_up) to (o);
\draw (other) to (other_up);
\coordinate (turn) at ($ (other) - (1.5,2) $);
\coordinate (upper_turn) at ($ (turn) + (-1,1) $);
\draw[out=270,in=0] (other) to (turn);
\draw[out=180,in=270] (turn) to (upper_turn);
\draw[out=90,in=180] (upper_turn) to (other_up);
\end{tikzpicture}
&
\tagarray{\label{eq:cutelim3}}
\end{tabular}
\end{center}
may be transformed to
\begin{center}
\begin{tabular}{ >{\centering}m{8cm} >{\centering}m{3cm} >{\centering}m{3cm}}
\AxiomC{$\pi_1$}
\noLine\UnaryInfC{$\vdots$}
\def\extraVskip{5pt}
\noLine\UnaryInfC{$\Gamma \vdash A$}
\def\extraVskip{2pt}
\AxiomC{$\pi_2$}
\noLine\UnaryInfC{$\vdots$}
\def\extraVskip{5pt}
\noLine\UnaryInfC{$A \vdash B$}
\def\extraVskip{2pt}
\RightLabel{\scriptsize cut}
\BinaryInfC{$\Gamma \vdash B$}
\AxiomC{$\pi_3$}
\noLine\UnaryInfC{$\vdots$}
\def\extraVskip{5pt}
\noLine\UnaryInfC{$B \vdash C$}
\def\extraVskip{2pt}
\RightLabel{\scriptsize cut}
\BinaryInfC{$\Gamma \vdash C$}
\DisplayProof
&
\begin{tikzpicture}[scale=0.35,auto,inner sep=1mm]
\node (top) at (0,5) {$C$};
\mapnode (pi3) at (0,3) {};
\node [right] at (pi3.east) {$\den{\pi_3}$};
\mapnode (pi2) at (0,1) {};
\node [right] at (pi2.east) {$\den{\pi_2}$};
\mapnode (pi1) at (0,-1) {};
\node [right] at (pi1.east) {$\den{\pi_1}$};
\node (bottom) at (0,-3) {$\Gamma$};
\draw (bottom) -- (pi1);
\draw (pi1) -- (pi2);
\draw (pi2) -- (pi3);
\draw (pi3) -- (top);
\end{tikzpicture}
&
\tagarray{\label{eq:cutelim4}}
\end{tabular}
\end{center}
Observe how this cut-elimination transformation encodes the meaning of the left introduction rule $\multimap L$, in the sense that it takes a proof $\pi_1$ of $\Gamma \vdash A$ and a proof $\pi_3$ of $B \vdash C$ and says: if you give me a proof of $A \vdash B$ I know how to sandwich it between $\pi_1$ and $\pi_3$.
\end{example}

\begin{example}\label{example:cut_elim_examples2} The cut-elimination transformation \cite[\S 3.11.10]{mellies} tells us that
\begin{center}
\begin{tabular}{ >{\centering}m{7cm} >{\centering}m{4cm} >{\centering}m{3cm}}
\AxiomC{$\pi_1$}
\noLine\UnaryInfC{$\vdots$}
\def\extraVskip{5pt}
\noLine\UnaryInfC{$\Gamma \vdash A$}
\def\extraVskip{2pt}
\AxiomC{$\pi_2$}
\noLine\UnaryInfC{$\vdots$}
\def\extraVskip{5pt}
\noLine\UnaryInfC{$B, A \vdash C$}
\def\extraVskip{2pt}
\RightLabel{\scriptsize $\multimap R$}
\UnaryInfC{$A \vdash B \multimap C$}
\RightLabel{\scriptsize cut}
\BinaryInfC{$\Gamma \vdash B \multimap C$}
\DisplayProof
&
\begin{tikzpicture}[inner sep = 0.5mm, scale=0.4,auto]
\node (topr) at (1,4) {$B \multimap C$};
\node[circle,draw=black,fill=black] (o) at (1,0) {};
\node [above] at (o.east) {$\;\;\;\;\;\;\;\;\den{\pi_2}$};
\node[circle,draw=black,fill=black] (other) at (2.5,-2) {};
\node [right] at (other.east) {$\den{\pi_1}$};
\node (gamma) at (2.5,-4) {$\Gamma$};
\draw[out=90,in=0] (gamma) to (o);
\draw[out=0,in=180] (-1.5,-2) to (o);
\draw[out=180,in=270] (-1.5,-2) to (-3,0);
\draw[out=90,in=180] (-3,0) to (1,2);
\draw (o) to (1,2);
\draw (1,2) to (topr);
\end{tikzpicture}
&
\tagarray{\label{eq:cutelim1}}
\end{tabular}
\end{center}
is transformed to the proof
\begin{center}
\begin{tabular}{ >{\centering}m{7cm} >{\centering}m{4cm} >{\centering}m{3cm}}
\AxiomC{$\pi_1$}
\noLine\UnaryInfC{$\vdots$}
\def\extraVskip{5pt}
\noLine\UnaryInfC{$\Gamma \vdash A$}
\def\extraVskip{2pt}
\AxiomC{$\pi_2$}
\noLine\UnaryInfC{$\vdots$}
\def\extraVskip{5pt}
\noLine\UnaryInfC{$B, A \vdash C$}
\def\extraVskip{2pt}
\RightLabel{\scriptsize cut}
\BinaryInfC{$B, \Gamma \vdash C$}
\RightLabel{\scriptsize $\multimap R$}
\UnaryInfC{$\Gamma \vdash B \multimap C$}
\DisplayProof
&
\begin{tikzpicture}[inner sep = 0.5mm, scale=0.4,auto]
\node (topr) at (1,4) {$B \multimap C$};
\node[circle,draw=black,fill=black] (o) at (1,0) {};
\node [above] at (o.east) {$\;\;\;\;\;\;\;\;\;\;\;\;\;\;\;\;\;\den{\pi_2} \circ \den{\pi_1}$};
\node (gamma) at (2.5,-4) {$\Gamma$};
\draw[out=90,in=0] (gamma) to (o);
\draw[out=0,in=180] (-1.5,-2) to (o);
\draw[out=180,in=270] (-1.5,-2) to (-3,0);
\draw[out=90,in=180] (-3,0) to (1,2);
\draw (o) to (1,2);
\draw (1,2) to (topr);
\end{tikzpicture}
&
\tagarray{\label{eq:cutelim2}}
\end{tabular}
\end{center}
and thus the two corresponding diagrams are equal. In this case, the equality generated by cut-elimination expresses the fact that the Hom-tensor adjunction is natural.
\end{example}

\begin{example}\label{example:cut_elim_examples3} The cut-elimination transformation \cite[\S 3.6.1]{mellies} tells us that
\begin{center}
\AxiomC{} 
\UnaryInfC{$A \vdash A$}
\AxiomC{$\pi$}
\noLine\UnaryInfC{$\vdots$}
\def\extraVskip{5pt}
\noLine\UnaryInfC{$A \vdash B$}
\RightLabel{\scriptsize cut}
\BinaryInfC{$A \vdash B$}
\DisplayProof
\end{center}
may be transformed to the proof $\pi$. Thus the axiom rule acts as an identity for cut.
\end{example}

The analogue in linear logic of a normal form in $\lambda$-calculus is

\begin{definition} A proof is \emph{cut-free} if it contains no occurrences of the cut rule.
\end{definition}

The main theorem is that the relation $\rightsquigarrow$ is \emph{weakly normalising}, which means:

\begin{theorem}[Cut-elimination] For each proof $\pi$ in linear logic there is a sequence
\begin{equation}\label{eq:proof_pi}
\pi = \pi_0 \rightsquigarrow \pi_1 \rightsquigarrow \cdots \rightsquigarrow \pi_n = \widetilde{\pi}
\end{equation}
where $\widetilde{\pi}$ is cut-free, and $n$ depends on $\pi$.
\end{theorem}
\begin{proof}
The original proof by Girard is given in the language of proof-nets \cite{girard_llogic}. For a proof in the language of sequent calculus see for example \cite[Appendix B]{brauner}. The argument follows Gentzen's proof for cut-elimination in classical logic; see \cite{gentzen} and \cite[Chapter 13]{girard_prooftypes}. 
\end{proof}

This is the analogue of Gentzen's \emph{Hauptsatz} for ordinary logic, which was mentioned in the introduction. At a high level, the sequence of rewrites \eqref{eq:proof_pi} eliminates the implicitness present in the original proof until the fully explicit, cut-free proof $\widetilde{\pi}$ is left. 

\begin{remark} It is possible that multiple rewrite rules apply to a given proof, so that there is more than one sequence \eqref{eq:proof_pi} with $\pi$ as its starting point. This raises the natural question: does \emph{every} sequence of rewrites starting from $\pi$ terminate with a cut-free proof? This property of a rewrite rule is called \emph{strong normalisation}. We could also ask whether it is the case that whenever a proof $\pi$ reduces to $\pi'$ as well as to $\pi''$, there exists a proof $\pi'''$ to which both of the proofs $\pi'$ and $\pi''$ reduce; this is called \emph{confluence} or the \emph{Church-Rosser property}. If a rewrite rule satisfies both strong normalisation and confluence then beginning with a proof $\pi$ we can apply arbitrarily chosen reductions and be assured that eventually this process terminates, and that the result is independent of all choices: that is, the result is a \emph{normal form} of $\pi$.

Unfortunately in its sequent calculus presentation linear logic does \emph{not} satisfy confluence and in particular it is not true that each cut-elimination equivalence class contains a unique cut-free proof \cite[\S 1.3.1]{girard_prooftypes}. However to some extent this is a ``mere'' technical problem, and is usually taken as evidence that the right presentation of linear logic is not via the (overly bureaucratic) sequent calculus but through \emph{proof-nets}. In Girard's original paper \cite[III.3]{girard_llogic} it is proven that cut-elimination for proof-nets satisfies strong normalisation in the current setting of first-order linear logic; regarding second-order linear logic and confluence the story is surprisingly intricate, see \cite{pagani}. Note also that the Natural Deduction presentation of linear logic is very close to the sequent calculus and satisfies both strong normalisation and confluence; see \cite{benton_strong} and \cite[\S 2.2]{brauner}.
\end{remark}

\subsection{An extended example}

To demonstrate cut-elimination in a nontrivial example, let us prove that $2 \times 2 = 4$.

\begin{definition}\label{definition:mult_m} Let $m \ge 0$ be an integer and $A$ a formula. We define $\prf{\mathrm{mult}}_A(m, -)$ to be the proof (writing $E = A \multimap A$)
\begin{equation}\label{cut_mult_2_0}
\begin{mathprooftree}
\AxiomC{$\church{m}_A$}
\noLine\UnaryInfC{$\vdots$}
\def\extraVskip{5pt}
\noLine\UnaryInfC{${!}E \vdash E$}
\def\extraVskip{2pt}
\RightLabel{\scriptsize prom}
\UnaryInfC{${!}E \vdash {!}E$}
\AxiomC{}
\UnaryInfC{$E \vdash E$}
\RightLabel{\scriptsize$\multimap L$}
\BinaryInfC{${!} E, \inta_A \vdash E$}
\RightLabel{\scriptsize$\multimap R$}
\UnaryInfC{$\inta_A \vdash \inta_A$}
\end{mathprooftree}
\end{equation}
\end{definition}

This proof represents multiplication by $m$, in the following sense.

\begin{lemma} The proof obtained from $\prf{\mathrm{mult}}_A(m,-)$ by cutting against $\church{n}_A$
\begin{equation}\label{cut_mult_prooftree}
\begin{mathprooftree}
\AxiomC{$\church{n}_A$}
\noLine\UnaryInfC{$\vdots$}
\def\extraVskip{5pt}
\noLine\UnaryInfC{$\vdash \inta_A$}
\def\extraVskip{2pt}
\AxiomC{$\prf{\mathrm{mult}}_A(m,-)$}
\noLine\UnaryInfC{$\vdots$}
\def\extraVskip{5pt}
\noLine\UnaryInfC{$\inta_A \vdash \inta_A$}
\def\extraVskip{2pt}
\RightLabel{\scriptsize cut}
\BinaryInfC{$\vdash \inta_A$}
\end{mathprooftree}
\end{equation}
is equivalent under cut-elimination to $\church{mn}_A$.
\end{lemma}

Rather than give a complete proof of this lemma, we examine the special case where $m = n = 2$. If we denote the proof in \eqref{cut_mult_prooftree} by $\prf{\mathrm{mult}}_A(2,-) \l \church{2}_A$ then its string diagram is
\begin{equation}\label{cut_mult_2}
\prf{\mathrm{mult}}_A(2,-) \l \church{2}_A \quad = \begin{tikzpicture}[scale=0.35,auto,inner sep=1mm,baseline=(current  bounding  box.center)]
\coordinate (topr) at (0,2);
\coordinate (comp) at (3,4);
\drawbang[out=90,in=180] (topr) to (comp);
\drawprom (0,-1) ellipse (3cm and 3cm);
\dernode (bottomr) at (0,-4) {};

\coordinate (o) at ($ (topr) + (1,-2) $);
\coordinate (ellipse_center) at ($ (o) - (1,1) $);
\coordinate (elbow) at ($ (o) - (3, 0.5) $);
\coordinate (top_of_circle) at ($ (o) + (-1,2) $);
\draw[out=90,in=270] (o) to (top_of_circle);
\drawprom (ellipse_center) ellipse (3cm and 3cm); 
\dernode (bottomr) at ($ (o) - (1,4) $) {};
\dernode (R) at ($ (o) + (1, -1.5) $) {}; 
\dernode (L) at ($ (o) - (1,1.5) $) {}; 
\coordinate (left_curve) at ($ (o) - (2.5,2) $);
\coordinate (left_meet) at ($ (o) - (1.5, 0.9) $);
\coordinate (delta) at ($ (o) - (0, 2.7) $); 
\draw[out=90,in=0] (R) to node [swap] {} (o);
\drawbang[out=0,in=270] (delta) to (R);
\drawbang[out=180,in=270] (delta) to (L);
\drawbang[out=90,in=270] (bottomr) to (delta);
\draw[out=90,in=180] (left_meet) to (o);
\draw[out=90,in=0] (L) to (left_meet);
\draw[out=0,in=180] (left_curve) to (left_meet);
\draw[out=180,in=270] (left_curve) to (elbow);
\draw[out=90,in=180] (elbow) to ($ (o)!.5!(top_of_circle) $);

\coordinate (curve) at (-3.9, -3);
\coordinate (meet) at (3,5.5);
\drawbang[out=300,in=270] (curve) to (bottomr);
\drawbang[out=120,in=180] (curve) to (meet);
\node (vtop) at (3,7) {$\inta_A$};
\draw (meet) to (vtop);
\draw (comp) to (meet);

\coordinate (2o) at ($ (topr) + (8,-2) $);

\coordinate (2left_meet) at ($ (2o) - (1.5, 0.9) $);
\draw[out=90,in=180] (2left_meet) to (2o);

\dernode (2R) at ($ (2o) + (1,-1.5) $) {};
\dernode (2L) at ($ (2o) + (-1,-1.5) $) {};
\coordinate (2delta) at ($ (2o) - (0,2.7) $);
\draw[out=90,in=0] (2R) to (2o);
\drawbang[out=0,in=270] (2delta) to (2R);
\drawbang[out=180,in=270] (2delta) to (2L);

\coordinate (2left_curve) at ($ (2o) - (2.5, 2) $);
\coordinate (2left_curve_mid) at ($ (2o) - (3,1.25) $);
\coordinate (2first_meeting_top) at ($ (2o) + (0,0.75) $);
\draw[out=90,in=0] (2L) to (2left_meet);
\draw[out=0,in=180] (2left_curve) to (2left_meet);
\draw[out=180,in=270] (2left_curve) to (2left_curve_mid);

\coordinate (2curve_bottom) at ($ (2left_meet) - (0,2.5) $);
\coordinate (2curve_left) at ($ (2o) - (3.75, 1) $);
\drawbang[out=0,in=270] (2curve_bottom) to (2delta);
\drawbang[out=180,in=270] (2curve_bottom) to (2curve_left);

\coordinate (2top) at ($ (2o) + (0,1.5) $);
\draw[out=90,in=0] (2top) to (comp);
\draw (2o) to (2top);
\drawbang[out=90,in=180] (2curve_left) to (2top);
\draw[out=90,in=180] (2left_curve_mid) to ($ (2o)!.5!(2top) $);
\end{tikzpicture}
\end{equation}
To prove that the cut-free normalisation of $\prf{\mathrm{mult}}_A(2,-) \l \church{2}_A$ is the Church numeral $\church{4}_A$ we run through the proof transformations generated by the cut-elimination algorithm, each of which yields a new proof with the same denotation. This sequence of proofs represents the following sequence of manipulations of the string diagram:
\begin{itemize}
\item The first reduction \eqref{eq:cutelim1} $\rightsquigarrow$ \eqref{eq:cutelim2} applies naturality of the Hom-tensor adjunction to \eqref{cut_mult_2}. Then we are in the position of \eqref{eq:cutelim3}, with $\pi_2$ a part of $\church{2}_A$ and $\pi_1$ the promoted Church numeral and the reduction \eqref{eq:cutelim3} $\rightsquigarrow$ \eqref{eq:cutelim4} takes the left leg of the diagram and feeds it as an input to the right leg.
\item Next, we use that a promotion box represents a morphism of coalgebras, and thus can be commuted past the coproduct whereby it is duplicated.
\item Then the promotions cancel with the derelictions by the identity \eqref{eq:prom_cancel_der}, ``releasing'' the pair of Church numerals contained in the promotion boxes.
\item Finally, we may apply the general form of \eqref{eq:cutelim3} $\rightsquigarrow$ \eqref{eq:cutelim4}.
\end{itemize}
Each of these steps corresponds to one of the equalities in the following chain of diagrams:
\begin{equation}\label{diagram_mult_2_2}
\prf{\mathrm{mult}}_A(2,-) \l \church{2}_A \ = \ 
\begin{tikzpicture}[scale=0.35,auto,inner sep=1mm,baseline=(current  bounding  box.center)]
\coordinate (topr) at (0,2); 
\coordinate (curve) at (-4, -2.9);
\node (vtop) at ($ (topr) + (0,8) $) {};
\coordinate (meet) at ($ (vtop) - (0,2) $); 

\draw (meet) to (vtop);

\coordinate (2o) at ($ (topr) + (0,4) $);
\coordinate (2left_meet) at ($ (2o) - (1.5, 0.9) $);
\draw[out=90,in=180] (2left_meet) to (2o);

\dernode (2R) at ($ (2o) + (1,-1.5) $) {};
\dernode (2L) at ($ (2o) + (-1,-1.5) $) {};
\coordinate (2delta) at ($ (2o) - (0,2.7) $);
\draw[out=90,in=0] (2R) to (2o);
\drawbang[out=0,in=270] (2delta) to (2R);
\drawbang[out=180,in=270] (2delta) to (2L);

\coordinate (2left_curve) at ($ (2o) - (2.5, 2) $);
\coordinate (2left_curve_mid) at ($ (2o) - (3,1.25) $);
\coordinate (2first_meeting_top) at ($ (2o) + (0,0.75) $);
\draw[out=90,in=0] (2L) to (2left_meet);
\draw[out=0,in=180] (2left_curve) to (2left_meet);
\draw[out=180,in=270] (2left_curve) to (2left_curve_mid);
\draw[out=90,in=180] (2left_curve_mid) to (2first_meeting_top);
\draw (2o) to (meet);

\drawbang (topr) to (2delta);

\coordinate (o) at ($ (topr) + (1,-2) $);
\coordinate (ellipse_center) at ($ (o) - (1,1) $);
\coordinate (elbow) at ($ (o) - (3, 0.5) $);
\coordinate (top_of_circle) at ($ (o) + (-1,2) $);
\draw[out=90,in=270] (o) to (top_of_circle);
\drawprom (ellipse_center) ellipse (3cm and 3cm); 
\dernode (bottomr) at ($ (o) - (1,4) $) {};
\dernode (R) at ($ (o) + (1, -1.5) $) {}; 
\dernode (L) at ($ (o) - (1,1.5) $) {}; 
\coordinate (left_curve) at ($ (o) - (2.5,2) $);
\coordinate (left_meet) at ($ (o) - (1.5, 0.9) $);
\coordinate (delta) at ($ (o) - (0, 2.7) $); 
\draw[out=90,in=0] (R) to node [swap] {} (o);
\drawbang[out=0,in=270] (delta) to (R);
\drawbang[out=180,in=270] (delta) to (L);
\drawbang[out=90,in=270] (bottomr) to (delta);
\draw[out=90,in=180] (left_meet) to (o);
\draw[out=90,in=0] (L) to (left_meet);
\draw[out=0,in=180] (left_curve) to (left_meet);
\draw[out=180,in=270] (left_curve) to (elbow);
\draw[out=90,in=180] (elbow) to ($ (o)!.5!(top_of_circle) $);

\drawbang[out=300,in=270] (curve) to (bottomr);
\drawbang[out=120,in=180] (curve) to (meet);
\end{tikzpicture}
\quad = \quad
\begin{tikzpicture}[scale=0.4,auto,inner sep=1mm,baseline=(current  bounding  box.center)]
\coordinate (topr) at (0,2); 
\node (vtop) at ($ (topr) + (0,8) $) {};
\coordinate (meet) at ($ (vtop) - (0,2) $); 

\draw (meet) to (vtop);

\coordinate (2o) at ($ (topr) + (0,4) $);
\coordinate (2left_meet) at ($ (2o) - (1.5, 0.9) $);
\draw[out=90,in=180] (2left_meet) to (2o);

\dernode (2R) at ($ (2o) + (1,-1.5) $) {};
\dernode (2L) at ($ (2o) + (-1,-1.5) $) {};
\draw[out=90,in=0] (2R) to (2o);

\coordinate (2left_curve) at ($ (2o) - (2.5, 2) $);
\coordinate (2left_curve_mid) at ($ (2o) - (3,1.25) $);
\coordinate (2first_meeting_top) at ($ (2o) + (0,0.75) $);
\draw[out=90,in=0] (2L) to (2left_meet);
\draw[out=0,in=180] (2left_curve) to (2left_meet);
\draw[out=180,in=270] (2left_curve) to (2left_curve_mid);
\draw[out=90,in=180] (2left_curve_mid) to (2first_meeting_top);
\draw (2o) to (meet);

\drawbang[out=90,in=270] ($ (topr) + (-3.5,0) $) to (2L);
\drawbang[out=90,in=270] ($ (topr) + (3.5,0) $) to (2R);

\coordinate (o) at ($ (topr) + (-2.5,-2) $);
\coordinate (ellipse_center) at ($ (o) - (1,1) $);
\coordinate (elbow) at ($ (o) - (3, 0.5) $);
\coordinate (top_of_circle) at ($ (o) + (-1,2) $);
\draw[out=90,in=270] (o) to (top_of_circle);
\drawprom (ellipse_center) ellipse (3cm and 3cm); 
\dernode (bottomr) at ($ (o) - (1,4) $) {};
\dernode (R) at ($ (o) + (1, -1.5) $) {}; 
\dernode (L) at ($ (o) - (1,1.5) $) {}; 
\coordinate (left_curve) at ($ (o) - (2.5,2) $);
\coordinate (left_meet) at ($ (o) - (1.5, 0.9) $);
\coordinate (delta) at ($ (o) - (0, 2.7) $); 
\draw[out=90,in=0] (R) to node [swap] {} (o);
\drawbang[out=0,in=270] (delta) to (R);
\drawbang[out=180,in=270] (delta) to (L);
\drawbang[out=90,in=270] (bottomr) to (delta);
\draw[out=90,in=180] (left_meet) to (o);
\draw[out=90,in=0] (L) to (left_meet);
\draw[out=0,in=180] (left_curve) to (left_meet);
\draw[out=180,in=270] (left_curve) to (elbow);
\draw[out=90,in=180] (elbow) to ($ (o)!.5!(top_of_circle) $);

\coordinate (o) at ($ (topr) + (4.5,-2) $);
\coordinate (ellipse_center) at ($ (o) - (1,1) $);
\coordinate (elbow) at ($ (o) - (3, 0.5) $);
\coordinate (top_of_circle) at ($ (o) + (-1,2) $);
\draw[out=90,in=270] (o) to (top_of_circle);
\drawprom (ellipse_center) ellipse (3cm and 3cm); 
\dernode (right_bottomr) at ($ (o) - (1,4) $) {};
\dernode (R) at ($ (o) + (1, -1.5) $) {}; 
\dernode (L) at ($ (o) - (1,1.5) $) {}; 
\coordinate (left_curve) at ($ (o) - (2.5,2) $);
\coordinate (left_meet) at ($ (o) - (1.5, 0.9) $);
\coordinate (delta) at ($ (o) - (0, 2.7) $); 
\draw[out=90,in=0] (R) to node [swap] {} (o);
\drawbang[out=0,in=270] (delta) to (R);
\drawbang[out=180,in=270] (delta) to (L);
\drawbang[out=90,in=270] (right_bottomr) to (delta);
\draw[out=90,in=180] (left_meet) to (o);
\draw[out=90,in=0] (L) to (left_meet);
\draw[out=0,in=180] (left_curve) to (left_meet);
\draw[out=180,in=270] (left_curve) to (elbow);
\draw[out=90,in=180] (elbow) to ($ (o)!.5!(top_of_circle) $);

\coordinate (curve) at (-8, 0);
\coordinate (very_bottom_turn) at (-2,-7);
\coordinate (very_bottom_delta) at (0,-6);
\drawbang[out=270,in=180] (curve) to (very_bottom_turn);
\drawbang[out=90,in=180] (curve) to (meet);
\drawbang[out=0,in=270] (very_bottom_turn) to (very_bottom_delta);
\drawbang[out=0,in=270] (very_bottom_delta) to (right_bottomr);
\drawbang[out=180,in=270] (very_bottom_delta) to (bottomr);
\end{tikzpicture}
\end{equation}
\begin{equation}\label{diagram_mult_2_3}
= \quad
\begin{tikzpicture}[scale=0.3,auto,inner sep=1mm,baseline=(current  bounding  box.center)]
\coordinate (topr) at (0,2); 
\node (vtop) at ($ (topr) + (0,8) $) {};
\coordinate (meet) at ($ (vtop) - (0,2) $); 

\draw (meet) to (vtop);

\coordinate (2o) at ($ (topr) + (0,4) $);
\coordinate (2left_meet) at ($ (2o) - (1.5, 0.9) $);
\draw[out=90,in=180] (2left_meet) to (2o);

\coordinate (2R) at ($ (2o) + (1,-1.5) $);
\coordinate (2L) at ($ (2o) + (-1,-1.5) $);
\draw[out=90,in=0] (2R) to (2o);

\coordinate (2left_curve) at ($ (2o) - (2.5, 2) $);
\coordinate (2left_curve_mid) at ($ (2o) - (3,1.25) $);
\coordinate (2first_meeting_top) at ($ (2o) + (0,0.75) $);
\draw[out=90,in=0] (2L) to (2left_meet);
\draw[out=0,in=180] (2left_curve) to (2left_meet);
\draw[out=180,in=270] (2left_curve) to (2left_curve_mid);
\draw[out=90,in=180] (2left_curve_mid) to (2first_meeting_top);
\draw (2o) to (meet);

\draw[out=90,in=270] ($ (topr) + (-3.5,0) $) to (2L);
\draw[out=90,in=270] ($ (topr) + (3.5,0) $) to (2R);

\coordinate (o) at ($ (topr) + (-2.5,-2) $);
\coordinate (elbow) at ($ (o) - (3, 0.5) $);
\coordinate (top_of_circle) at ($ (o) + (-1,2) $);
\dernode (R) at ($ (o) + (1, -1.8) $) {}; 
\dernode (L) at ($ (o) - (1,1.8) $) {}; 
\coordinate (left_curve) at ($ (o) - (2.5,2) $);
\coordinate (left_meet) at ($ (o) - (1.5, 0.9) $);
\coordinate (delta) at ($ (o) - (0, 3) $); 
\draw[out=90,in=0] (R) to node [swap] {} (o);
\drawbang[out=0,in=270] (delta) to (R);
\drawbang[out=180,in=270] (delta) to (L);
\draw[out=90,in=180] (left_meet) to (o);
\draw[out=90,in=0] (L) to (left_meet);
\draw[out=0,in=180] (left_curve) to (left_meet);
\draw[out=180,in=270] (left_curve) to (elbow);
\draw[out=90,in=270] (o) to (top_of_circle);
\draw[out=90,in=180] (elbow) to ($ (o)!.5!(top_of_circle) $);

\coordinate (o) at ($ (topr) + (4.5,-2) $);
\coordinate (elbow) at ($ (o) - (3, 0.5) $);
\coordinate (top_of_circle) at ($ (o) + (-1,2) $);
\dernode (R) at ($ (o) + (1, -1.8) $) {}; 
\dernode (L) at ($ (o) - (1,1.8) $) {}; 
\coordinate (left_curve) at ($ (o) - (2.5,2) $);
\coordinate (left_meet) at ($ (o) - (1.5, 0.9) $);
\coordinate (right_delta) at ($ (o) - (0, 3) $); 
\draw[out=90,in=0] (R) to node [swap] {} (o);
\drawbang[out=0,in=270] (right_delta) to (R);
\drawbang[out=180,in=270] (right_delta) to (L);
\draw[out=90,in=180] (left_meet) to (o);
\draw[out=90,in=0] (L) to (left_meet);
\draw[out=0,in=180] (left_curve) to (left_meet);
\draw[out=180,in=270] (left_curve) to (elbow);
\draw[out=90,in=270] (o) to (top_of_circle);
\draw[out=90,in=180] (elbow) to ($ (o)!.5!(top_of_circle) $);

\coordinate (curve) at (-8, 0);
\coordinate (very_bottom_turn) at (-2,-7);
\coordinate (very_bottom_delta) at (1,-5);
\drawbang[out=270,in=180] (curve) to (very_bottom_turn);
\drawbang[out=90,in=180] (curve) to (meet);
\drawbang[out=0,in=270] (very_bottom_turn) to (very_bottom_delta);
\drawbang[out=0,in=270] (very_bottom_delta) to (right_delta);
\drawbang[out=180,in=270] (very_bottom_delta) to (delta);
\end{tikzpicture}
\quad = \quad
\begin{tikzpicture}[scale=0.3,auto,inner sep=1mm,baseline=(current  bounding  box.center)]
\coordinate (topr) at (0,2); 
\node (vtop) at ($ (topr) + (0,8) $) {};
\coordinate (meet) at ($ (vtop) - (0,2) $); 

\draw (meet) to (vtop);

\coordinate (o) at ($ (topr) + (-2.5,-2) $);
\coordinate (elbow) at ($ (o) - (3, 0.5) $);
\coordinate (top_of_circle) at ($ (o) + (0,1) $);
\dernode (R) at ($ (o) + (1, -1.8) $) {}; 
\dernode (L) at ($ (o) - (1,1.8) $) {}; 
\coordinate (left_curve) at ($ (o) - (2.5,2) $);
\coordinate (left_meet) at ($ (o) - (1.5, 0.9) $);
\coordinate (delta) at ($ (o) - (0, 3) $); 
\draw[out=90,in=0] (R) to node [swap] {} (o);
\drawbang[out=0,in=270] (delta) to (R);
\drawbang[out=180,in=270] (delta) to (L);
\draw[out=90,in=180] (left_meet) to (o);
\draw[out=90,in=0] (L) to (left_meet);
\draw[out=0,in=180] (left_curve) to (left_meet);
\draw[out=180,in=270] (left_curve) to (elbow);
\draw[out=90,in=270] (o) to (top_of_circle);
\draw[out=90,in=180] (elbow) to (top_of_circle);

\coordinate (o) at ($ (topr) + (2.5,3) $);
\coordinate (elbow) at ($ (o) - (3, 0.5) $);
\coordinate (right_top_of_circle) at ($ (o) + (-1,2) $);
\dernode (R) at ($ (o) + (1, -1.8) $) {}; 
\dernode (L) at ($ (o) - (1,1.8) $) {}; 
\coordinate (left_curve) at ($ (o) - (2.5,2) $);
\coordinate (left_meet) at ($ (o) - (1.5, 0.9) $);
\coordinate (right_delta) at ($ (o) - (0, 3) $); 
\draw[out=90,in=0] (R) to node [swap] {} (o);
\drawbang[out=0,in=270] (right_delta) to (R);
\drawbang[out=180,in=270] (right_delta) to (L);
\draw[out=90,in=180] (left_meet) to (o);
\draw[out=90,in=0] (L) to (left_meet);

\draw[out=90,in=180] (top_of_circle) to (left_meet);
\draw[out=90,in=270] (o) to (meet);

\coordinate (curve) at (-8, 0);
\coordinate (very_bottom_turn) at (-2,-7);
\coordinate (very_bottom_delta) at (0,-5);
\drawbang[out=270,in=180] (curve) to (very_bottom_turn);
\drawbang[out=90,in=180] (curve) to (meet);
\drawbang[out=0,in=270] (very_bottom_turn) to (very_bottom_delta);
\drawbang[out=0,in=270] (very_bottom_delta) to (right_delta);
\drawbang[out=180,in=270] (very_bottom_delta) to (delta);
\end{tikzpicture}
\end{equation}
This last diagram is the denotation of $\church{4}_A$, so we conclude that (at least at the level of the denotations) the output of the program $\prf{\mathrm{mult}}_A(2,-)$ on the input $\church{2}_A$ is $\church{4}_A$. The cut-elimination transformations corresponding to these diagram manipulations are given in Appendix \ref{section:appendix_cut_elim}.

\section{Second-order linear logic}\label{section:second}

A measure of the strength or expressiveness of a logic is the class of functions $\mathbb{N} \lto \mathbb{N}$ that can be encoded as proofs. In first-order intuitionistic linear logic $\ILL$ we have already seen how to encode multiplication by an integer as a proof (Definition \ref{definition:mult_m}) and addition will be addressed below (Example \ref{example:addition}) so we deduce that for any formula $A$, any polynomial function $f: \mathbb{N} \lto \mathbb{N}$ may be encoded as a proof of
\begin{equation}\label{eq:intatointa}
{!}\inta_A \vdash \inta_A\,.
\end{equation}
More precisely, for any polynomial function $f$ there exists a proof $F$ of the above sequent with the property that for any integer $n \ge 0$, the cut of $F$ against $\church{n}_A$ reduces under cut-elimination to $\underline{f(n)}_A$. This is close to being a complete list of the functions $f$ that can be encoded as proofs of the sequent \eqref{eq:intatointa}, which comes as a bit of a disappointment; see \cite{schwicht} and \cite{zakrzewski}. However, the good news is that more expressive power can be obtained if we allow ourselves to use proofs of $\inta_B \vdash \inta_A$ where $B$ is somehow constructed from $A$. For example, we will see below how to encode exponentials using $B = \inta_A$. 

A very expressive system can be obtained by allowing the base type $B$ to be computed from $A$ at ``run time'', that is, during cut-elimination. This is achieved by adding quantifiers to the language of $\ILL$, and the resulting logic is called \emph{second-order intuitionistic linear logic} $\ILL_2$. One indication of the expressiveness of this extension is that System F \cite[\S 11]{girard_prooftypes} otherwise known as the polymorphic $\lambda$-calculus, embeds into $\ILL_2$ \cite[\S 5.2]{girard_llogic}. These systems have a growing influence on practical programming languages: for example the language Haskell compiles internally to an extension of System F. 
\\


Our aim in this section is to begin by writing down some more arithmetic in $\ILL$, and then to show how iteration on higher types leads naturally to the need for quantifiers in order to properly represent towers of exponentials. The semantics of $\ILL_2$ is much less clear than the propositional case so we will not speak about it here, but see \cite{girard_systemf,seely_systemf, seely}.

The following examples are taken from \cite[\S 5.3.2]{girard_llogic} and \cite[\S 4]{danos}.

\begin{definition} We define $\prf{\mathrm{comp}}_A$ to be the proof 
\begin{center}
\AxiomC{}
\UnaryInfC{$A \vdash A$}
\AxiomC{}
\UnaryInfC{$A \vdash A$}
\AxiomC{}
\UnaryInfC{$A \vdash A$}
\RightLabel{\scriptsize$\multimap L$}
\BinaryInfC{$A, A \multimap A \vdash A$}
\RightLabel{\scriptsize$\multimap L$}
\BinaryInfC{$A, A \multimap A, A \multimap A \vdash A$}
\RightLabel{\scriptsize$\multimap R$}
\UnaryInfC{$A \multimap A, A \multimap A \vdash A \multimap A$}
\DisplayProof
\qquad
\tagarray{\label{proof_comp}}
\end{center}
As discussed in Remark \ref{example:church_2} the denotation of this proof is the function which \emph{composes} two endomorphisms of  $\den{A}$.
\end{definition}

\begin{example}[(Addition)]\label{example:addition} We define $\prf{\mathrm{add}}_A$ to be the proof (writing $E = A \multimap A$)
\begin{center}
\AxiomC{}
\UnaryInfC{${!}E \vdash {!}E$}
\AxiomC{}
\UnaryInfC{${!}E \vdash {!}E$}
\AxiomC{$\prf{\mathrm{comp}}_A$}
\noLine\UnaryInfC{$\vdots$}
\def\extraVskip{5pt}
\noLine\UnaryInfC{$E, E \vdash E$}
\RightLabel{\scriptsize$\multimap L$}
\BinaryInfC{${!}E, E, \inta_A \vdash E$}
\RightLabel{\scriptsize$\multimap L$}
\BinaryInfC{${!} E, {!} E, \inta_A, \inta_A \vdash E$}
\RightLabel{\scriptsize ctr}
\UnaryInfC{${!} E, \inta_A, \inta_A \vdash E$}
\RightLabel{\scriptsize$\multimap R$}
\UnaryInfC{$\inta_A, \inta_A \vdash \inta_A$}
\DisplayProof
\qquad
\tagarray{\label{add_prooftree}}
\end{center}
which encodes addition on $A$-integers in the following sense: if $\prf{\mathrm{add}}_A$ is cut against two proofs $\church{m}_A$ and $\church{n}_A$ the resulting proof is equivalent under cut-elimination to $\church{m+n}_A$. 

Let us call the result of the cut $\prf{\mathrm{add}}_A( m, n )$, which is a proof of $\inta_A$. One way to see that this reduces to $\underline{m+n}_A$ without laboriously performing cut-elimination by hand is to use a term calculus, as presented in for example \cite{abramsky, benton_etal} or \cite[\S 4]{danos}, to understand the computational content of the proof. Alternatively, with the vector space semantics in mind, one can understand the proof as follows: given a vacuum $\vacu_\alpha$ at an endomorphism $\alpha$ of $V = \den{A}$, the contraction step duplicates this to $\vacu_\alpha \otimes \vacu_\alpha$. The left hand copy of $\vacu_\alpha$ is fed into $\underline{m}_A$ yielding $\alpha^m$ and the right hand copy is fed into $\underline{n}_A$ yielding $\alpha^n$. Then the composition ``machine'' composes these outputs to yield $\alpha^{m+n}$.
\end{example}

To generate a hierarchy of increasingly complex proofs from addition and multiplication we employ \emph{iteration}.

\begin{example}[(Iteration)]\label{example:proof_iteration} Given a proof $\beta$ of $A$ and $\beta'$ of $A \multimap A$ we define $\prf{\mathrm{rec}}_A(\beta, \beta')$ to be the proof
\begin{center}
\AxiomC{$\beta'$}
\noLine\UnaryInfC{$\vdots$}
\def\extraVskip{5pt}
\noLine\UnaryInfC{$\vdash A \multimap A$}
\RightLabel{\scriptsize prom}
\UnaryInfC{$\vdash {!}(A \multimap A)$}
\def\extraVskip{2pt}
\AxiomC{$\beta$}
\noLine\UnaryInfC{$\vdots$}
\def\extraVskip{5pt}
\noLine\UnaryInfC{$\vdash A$}
\AxiomC{}
\UnaryInfC{$A \vdash A$}
\RightLabel{\scriptsize$\multimap L$}
\BinaryInfC{$A \multimap A \vdash A$}
\RightLabel{\scriptsize$\multimap L$}
\BinaryInfC{$\inta_A \vdash A$}
\DisplayProof
\qquad
\tagarray{\label{iteration_prooftree}}
\end{center}
Cut against $\church{n}_A$ this is equivalent under cut-elimination to the $n$th power of $\beta'$ applied to $\beta$, that is, $n$ copies of $\beta'$ cut against one another and then cut against $\beta$.
\end{example}

This game gets more interesting when we get to exponentials. From Definition \ref{definition:mult_m} we know how to define a proof $\prf{\mathrm{mult}}_A(m,-)$ whose cut against $\church{n}_A$ yields something equivalent under cut-elimination to $\church{mn}_A$. But how do we construct a proof which, when cut against $\church{n}_A$, yields a proof equivalent to $\church{m^n}_A$? Naturally we can iterate multiplication by $m$, but the catch is that this requires integers of type $\inta_{\inta_A}$, as we will see in the next example.

\begin{example}[(Exponentials)]\label{example:exponentials} We define
\[
\prf{\mathrm{exp}}_{A,m} = \prf{\mathrm{rec}}_{\inta_A}\big( \church{1}_{A}, \prf{\mathrm{mult}}_A(m,-) \big)\,.
\]
That is,
\begin{center}
\AxiomC{$\prf{\mathrm{mult}}_A(m,-)$}
\noLine\UnaryInfC{$\vdots$}
\def\extraVskip{5pt}
\noLine\UnaryInfC{$\vdash \inta_A \multimap \inta_A$}
\RightLabel{\scriptsize prom}
\UnaryInfC{$\vdash {!}(\inta_A \multimap \inta_A)$}
\def\extraVskip{2pt}
\AxiomC{$\church{1}_{A}$}
\noLine\UnaryInfC{$\vdots$}
\def\extraVskip{5pt}
\noLine\UnaryInfC{$\vdash \inta_A$}
\AxiomC{}
\UnaryInfC{$\inta_A \vdash \inta_A$}
\RightLabel{\scriptsize$\multimap L$}
\BinaryInfC{$\inta_A \multimap \inta_A \vdash \inta_A$}
\RightLabel{\scriptsize$\multimap L$}
\BinaryInfC{$\inta_{\inta_A} \vdash \inta_A$}
\DisplayProof
\qquad
\tagarray{\label{exp_prooftree}}
\end{center}
Cut against $\church{n}_{\inta_A}$ this yields the desired numeral $\church{m^n}_A$.
\end{example}

We begin to see the general pattern: to represent more complicated functions $\mathbb{N} \lto \mathbb{N}$ we need to use more complicated base types $B$ for the integers $\inta_B$ on the left hand side. At the next step, when we try to iterate exponentials, we see that it is hopeless to continue without introducing a way to parametrise over these base types, and this is the role of quantifiers in second-order logic.

To motivate the extension of linear logic to second-order, consider the iteration of the function $n \mapsto 2^n$ which yields a tower of exponentials of variable height. More precisely, $E: \mathbb{N} \lto \mathbb{N}$ is defined recursively by $E(0) = 1, E(n+1) = 2^{E(n)}$ so that
\begin{equation}\label{eq:tower_of_exps_E}
E(n) = 2^{2^{2^{\iddots^{2}}}}
\end{equation}
where the total number of occurrences of $2$ on the right hand side is $n$. To represent the function $E$ as a proof in linear logic we need to introduce integer types with iterated subscripts by the recursive definition
\[
\inta_A(0) = \inta_A, \qquad \inta_A(r+1) =\inta_{\inta_A(r)}\,.
\]
By iteration of the exponential we mean the cut of the following sequence:
\begin{center}
\begin{tabular}{ >{\centering}m{3cm} >{\centering}m{2cm} >{\centering}m{3cm} >{\centering}m{3cm} >{\centering}m{2cm}}
\AxiomC{$\prf{\mathrm{exp}}_{\inta_A(n-1),2}$}
\noLine\UnaryInfC{$\vdots$}
\def\extraVskip{5pt}
\noLine\UnaryInfC{$\inta_A(n+1) \vdash \inta_A(n)$}
\DisplayProof

&

$\cdots$

&

\AxiomC{$\prf{\mathrm{exp}}_{\inta_{A},2}$}
\noLine\UnaryInfC{$\vdots$}
\def\extraVskip{5pt}
\noLine\UnaryInfC{$\inta_{\inta_{\inta_A}} \vdash \inta_{\inta_A}$}
\DisplayProof

&

\AxiomC{$\prf{\mathrm{exp}}_{A,2}$}
\noLine\UnaryInfC{$\vdots$}
\def\extraVskip{5pt}
\noLine\UnaryInfC{$\inta_{\inta_A} \vdash \inta_A$}
\DisplayProof

&
\tagarray{\label{seriesofexponentials}}
\end{tabular}
\end{center}
This yields a proof of $\inta_A(n+1) \vdash \inta_A$ which, when cut against $\church{n}_{\inta_A(n)}$, yields a proof equivalent under cut-elimination to $\church{E(n)}_A$. However this proof is constructed ``by hand'' for each integer $n$. It is clear that what we are doing in essence is iterating the exponential function, but we cannot express this formally in the language of propositional linear logic because the base type on our integers changes from iteration to iteration. To resolve this dilemma we will have to enrich the language by adding quantifiers, after which we will return in Lemma \ref{lemma:hypexp} to the problem of encoding towers of exponentials as proofs.

\begin{definition}[(Second-order linear logic)] The formulas of \emph{second-order linear logic} are defined recursively as follows: any formula of propositional linear logic is a formula, and if $A$ is a formula then so is $\forall x \,.\, A$ for any propositional variable $x$. There are two new deduction rules:
\begin{center}
\begin{tabular}{ >{\centering}m{4cm} >{\centering}m{4cm} }
\AxiomC{$\Gamma \vdash A$}
\RightLabel{\scriptsize$\forall R$}
\UnaryInfC{$\Gamma \vdash \forall x \,.\, A$}
\DisplayProof

&

\AxiomC{$\Gamma, A[B/x] \vdash C$}
\RightLabel{\scriptsize$\forall L$}
\UnaryInfC{$\Gamma, \forall x \,.\, A \vdash C$}
\DisplayProof
\end{tabular}
\end{center}
where $\Gamma$ is a sequence of formulas, possibly empty, and in the right introduction rule we require that $x$ is not free in any formula of $\Gamma$. Here $A[B/x]$ means a formula $A$ with all free occurrences of $x$ replaced by a formula $B$ (as usual, there is some chicanery necessary to avoid free variables in $B$ being captured, but we ignore this).
\end{definition}

Intuitively, the right introduction rule takes a proof $\pi$ and ``exposes'' the variable $x$, for which any type may be substituted. The left introduction rule, dually, takes a formula $B$ in some proof and binds it to a variable $x$. The result of cutting a right introduction rule against a left introduction rule is that the formula $B$ will be bound to $x$ throughout the proof $\pi$. That is, cut-elimination transforms
\begin{center}
\AxiomC{$\pi$}
\noLine\UnaryInfC{$\vdots$}
\def\extraVskip{5pt}
\noLine\UnaryInfC{$\Gamma \vdash A$}
\RightLabel{\scriptsize $\forall R$}
\UnaryInfC{$\Gamma \vdash \forall x \,.\, A$}
\def\extraVskip{2pt}
\AxiomC{$\rho$}
\noLine\UnaryInfC{$\vdots$}
\def\extraVskip{5pt}
\noLine\UnaryInfC{$\Delta, A[B/x] \vdash C$}
\RightLabel{\scriptsize $\forall L$}
\UnaryInfC{$\Delta, \forall x \,.\, A \vdash C$}
\RightLabel{\scriptsize cut}
\BinaryInfC{$\Gamma, \Delta \vdash C$}
\DisplayProof
\qquad
\tagarray{\label{quantifier_before}}
\end{center}
to the proof
\begin{center}
\AxiomC{$\pi[B/x]$}
\noLine\UnaryInfC{$\vdots$}
\def\extraVskip{5pt}
\noLine\UnaryInfC{$\Gamma \vdash A[B/x]$}
\def\extraVskip{2pt}
\AxiomC{$\rho$}
\noLine\UnaryInfC{$\vdots$}
\def\extraVskip{5pt}
\noLine\UnaryInfC{$\Delta, A[B/x] \vdash C$}
\RightLabel{\scriptsize cut}
\BinaryInfC{$\Gamma, \Delta \vdash C$}
\DisplayProof
\qquad
\tagarray{\label{quantifier_after}}
\end{center}
where $\pi[B/x]$ denotes the result of replacing all occurrences of $x$ in the proof $\pi$ with $B$. In the remainder of this section we provide a taste of just what an \emph{explosion} of possibilities the addition of quantifiers to the language represents. 

\begin{example}[(Integers)] The type of integers is
\[
\inta = \forall x \,.\, {!}(x \multimap x) \multimap (x \multimap x)\,.
\]
For each integer $n \ge 0$ we define $\church{n}$ to be the proof
\begin{center}
\AxiomC{$\church{n}_x$}
\noLine\UnaryInfC{$\vdots$}
\def\extraVskip{5pt}
\noLine\UnaryInfC{$\vdash \inta_x$}
\RightLabel{\scriptsize$\forall R$}
\UnaryInfC{$\vdash \inta$}
\DisplayProof
\end{center}
\end{example}


\begin{example}[(Exponentials)] We define $\prf{\mathrm{exp}}_m$ to be
\begin{center}
\AxiomC{$\prf{\mathrm{exp}}_{x,m}$}
\noLine\UnaryInfC{$\vdots$}
\def\extraVskip{5pt}
\noLine\UnaryInfC{$\inta_{\inta_x} \vdash \inta_x$}
\RightLabel{\scriptsize$\forall L$}
\UnaryInfC{$\inta \vdash \inta_x$}
\RightLabel{\scriptsize$\forall R$}
\UnaryInfC{$\inta \vdash \inta$}
\DisplayProof
\end{center}
\end{example}

\begin{example}[(Hyper-exponentials)] We define $\prf{\mathrm{hypexp}}$ to be
\begin{center}
\AxiomC{$\prf{\mathrm{exp}}_2$}
\noLine\UnaryInfC{$\vdots$}
\def\extraVskip{5pt}
\noLine\UnaryInfC{$\inta \vdash \inta$}
\RightLabel{\scriptsize $\multimap R$}
\UnaryInfC{$\vdash \inta \multimap \inta$}
\RightLabel{\scriptsize prom}
\UnaryInfC{$\vdash {!}(\inta \multimap \inta)$}
\def\extraVskip{2pt}
\AxiomC{$\church{1}$}
\noLine\UnaryInfC{$\vdots$}
\def\extraVskip{5pt}
\noLine\UnaryInfC{$\vdash \inta$}
\AxiomC{}
\UnaryInfC{$\inta \vdash \inta$}
\RightLabel{\scriptsize$\multimap L$}
\BinaryInfC{$\inta \multimap \inta \vdash \inta$}
\RightLabel{\scriptsize$\multimap L$}
\BinaryInfC{$\inta_{\inta} \vdash \inta$}
\RightLabel{\scriptsize $\forall L$}
\UnaryInfC{$\inta \vdash \inta$}
\DisplayProof
\qquad
\tagarray{\label{hyperexp_prooftree}}
\end{center}
\end{example}

\begin{lemma}\label{lemma:hypexp} The cut of $\church{n}$ against $\prf{\mathrm{hypexp}}$ reduces to $\church{E(n)}$.
\end{lemma}
\begin{proof}
We sketch how $\prf{\mathrm{hypexp}} \l \church{2}$ reduces to the sequence of cuts in \eqref{seriesofexponentials}, since the argument for general $n$ is similar. The first reduction is of the cut of left and right introduction rules for the quantifier \eqref{quantifier_before} $\rightsquigarrow$ \eqref{quantifier_after} which leaves a cut of $\church{2}_{\inta}$ against the proof \eqref{hyperexp_prooftree} up to its penultimate step. The second reduction is \eqref{eq:cutelim3} $\rightsquigarrow$ \eqref{eq:cutelim4}, i.e. \cite[\S 3.8.2]{mellies}, to
\begin{center}
\AxiomC{$\prf{\mathrm{exp}}_2$}
\noLine\UnaryInfC{$\vdots$}
\def\extraVskip{5pt}
\noLine\UnaryInfC{$\inta \vdash \inta$}
\RightLabel{\scriptsize $\multimap R$}
\UnaryInfC{$\vdash \inta \multimap \inta$}
\RightLabel{\scriptsize prom}
\UnaryInfC{$\vdash {!}(\inta \multimap \inta)$}
\def\extraVskip{2pt}
\AxiomC{$\church{2}_{\inta}$}
\noLine\UnaryInfC{$\vdots$}
\def\extraVskip{5pt}
\noLine\UnaryInfC{${!}( \inta \multimap \inta) \vdash \inta \multimap \inta$}
\RightLabel{\scriptsize cut}
\BinaryInfC{$\vdash \inta \multimap \inta$}
\AxiomC{$\church{1}$}
\noLine\UnaryInfC{$\vdots$}
\def\extraVskip{5pt}
\noLine\UnaryInfC{$\vdash \inta$}
\AxiomC{}
\UnaryInfC{$\inta \vdash \inta$}
\RightLabel{\scriptsize$\multimap L$}
\BinaryInfC{$\inta \multimap \inta \vdash \inta$}
\RightLabel{\scriptsize cut}
\BinaryInfC{$\vdash \inta$}
\DisplayProof
\quad
\tagarray{\label{hyperexp_prooftree_hyp}}
\end{center}
The left hand branch of this proof is familiar from Example \ref{example:cutagainst2}, so we know it reduces to the square of $\prf{\mathrm{exp}}_2$ which we can write as
\begin{center}
\AxiomC{$\prf{\mathrm{exp}}_{y,2}$}
\noLine\UnaryInfC{$\vdots$}
\def\extraVskip{5pt}
\noLine\UnaryInfC{$\inta_{\inta_y} \vdash \inta_y$}
\RightLabel{\scriptsize$\forall L$}
\UnaryInfC{$\inta \vdash \inta_y$}
\RightLabel{\scriptsize$\forall R$}
\UnaryInfC{$\inta \vdash \inta$}
\def\extraVskip{2pt}
\AxiomC{$\prf{\mathrm{exp}}_{x,2}$}
\noLine\UnaryInfC{$\vdots$}
\def\extraVskip{5pt}
\noLine\UnaryInfC{$\inta_{\inta_x} \vdash \inta_x$}
\RightLabel{\scriptsize$\forall L$}
\UnaryInfC{$\inta \vdash \inta_x$}
\RightLabel{\scriptsize$\forall R$}
\UnaryInfC{$\inta \vdash \inta$}
\RightLabel{\scriptsize cut}
\BinaryInfC{$\inta \vdash \inta$}
\DisplayProof
\quad
\tagarray{\label{hyperexp_quant}}
\end{center}
followed by a right introduction rule to get a proof of $\vdash \inta \multimap \inta$. There is a cut-elimination transformation that allows us to commute the cut past the right introduction rule in the right branch. At that point the right introduction rule (in the left branch) is cut against the left introduction rule (in the right branch) and the rewrite \eqref{quantifier_before} $\rightsquigarrow$ \eqref{quantifier_after} transforms this to a substitution of $y = \inta_x$ in the left branch:
\begin{center}
\AxiomC{$\prf{\mathrm{exp}}_{\inta_x,2}$}
\noLine\UnaryInfC{$\vdots$}
\def\extraVskip{5pt}
\noLine\UnaryInfC{$\inta_{\inta_{\inta_x}} \vdash \inta_{\inta_x}$}
\RightLabel{\scriptsize$\forall L$}
\UnaryInfC{$\inta \vdash \inta_{\inta_x}$}
\def\extraVskip{2pt}
\AxiomC{$\prf{\mathrm{exp}}_{x,2}$}
\noLine\UnaryInfC{$\vdots$}
\def\extraVskip{5pt}
\noLine\UnaryInfC{$\inta_{\inta_x} \vdash \inta_x$}
\RightLabel{\scriptsize cut}
\BinaryInfC{$\inta \vdash \inta_x$}
\RightLabel{\scriptsize$\forall R$}
\UnaryInfC{$\inta \vdash \inta$}
\DisplayProof
\quad
\tagarray{\label{hyperexp_quant2}}
\end{center}
This is enough to show that the proof $\prf{\mathrm{hypexp}} \l \church{2}$ reduces to the first two cuts in \eqref{seriesofexponentials}.
\end{proof}

Once we have added quantifiers, the expressive power of the language is immense. We can for example easily iterate hyper-exponentials to obtain a tower of exponentials whose height is itself a tower of exponentials, and by iterating at the type $\inta \multimap \inta$ obtain monsters like the Ackermann function \cite[\S 6.D]{girard_blindspot}. More precisely:

\begin{theorem}[Girard] Any recursive function $\mathbb{N} \lto \mathbb{N}$ which is provably total in second-order Peano arithmetic can be encoded into second-order linear logic as a proof of the sequent $\inta \vdash \inta$.
\end{theorem}
\begin{proof}
See \cite[\S 15.2]{girard_prooftypes} and \cite[\S 5.3.2]{girard_llogic}.
\end{proof}

\appendix

\section{Example of cut-elimination}\label{section:appendix_cut_elim}

We examine the beginning of the cut-elimination process applied to the proof \eqref{cut_mult_2}. Our reference for cut-elimination is Melli\`{e}s \cite[\S 3.3]{mellies}. Throughout a formula $A$ is fixed and we write $E = A \multimap A$ so that $\inta_A = {!}E \multimap E$. We encourage the reader to put the following series of proof trees side-by-side with the evolving diagrams in \eqref{cut_mult_2}-\eqref{diagram_mult_2_3} to see the correspondence between cut-elimination and diagram manipulation.

To begin, we expose the first layer of structure within $\prf{\mathrm{mult}}_A(2,-)$ to obtain
\begin{center}
\AxiomC{$\church{2}_A$}
\noLine\UnaryInfC{$\vdots$}
\def\extraVskip{5pt}
\noLine\UnaryInfC{$\vdash \inta_A$}
\def\extraVskip{2pt}
\AxiomC{$\prf{\mathrm{mult}}_A(2,-)$}
\noLine\UnaryInfC{$\vdots$}
\def\extraVskip{5pt}
\noLine\UnaryInfC{$!E, \inta_A \vdash E$}
\def\extraVskip{2pt}
\RightLabel{\scriptsize $\multimap R$}
\UnaryInfC{$\inta_A \vdash \inta_A$}
\RightLabel{\scriptsize cut}
\BinaryInfC{$\vdash \inta_A$}
\DisplayProof
\qquad
\tagarray{\label{cut_step_1}}
\end{center}
For a cut against a proof whose last deduction rule is a right introduction rule for $\multimap$, the cut elimination procedure \cite[\S 3.11.10]{mellies} prescribes that \eqref{cut_step_1} be transformed to
\begin{center}
\AxiomC{$\church{2}_A$}
\noLine\UnaryInfC{$\vdots$}
\def\extraVskip{5pt}
\noLine\UnaryInfC{$\vdash \inta_A$}
\def\extraVskip{2pt}
\AxiomC{$\prf{\mathrm{mult}}_A(2,-)$}
\noLine\UnaryInfC{$\vdots$}
\def\extraVskip{5pt}
\noLine\UnaryInfC{$!E, \inta_A \vdash E$}
\def\extraVskip{2pt}
\RightLabel{\scriptsize cut}
\BinaryInfC{$!E \vdash E$}
\RightLabel{\scriptsize $\multimap R$}
\UnaryInfC{$\vdash \inta_A$}
\DisplayProof
\qquad
\tagarray{\label{cut_step_2}}
\end{center}
If we fill in the content of $\prf{\mathrm{mult}}_A(2,-)$, this proof may be depicted as follows:
\begin{center}
\AxiomC{$\church{2}_A$}
\noLine\UnaryInfC{$\vdots$}
\def\extraVskip{5pt}
\noLine\UnaryInfC{$!E  \vdash E$}
\def\extraVskip{2pt}
\UnaryInfC{$\vdash \inta_A$}
\AxiomC{$\church{2}_A$}
\noLine\UnaryInfC{$\vdots$}
\def\extraVskip{5pt}
\noLine\UnaryInfC{$!E \vdash E$}
\def\extraVskip{2pt}
\RightLabel{\scriptsize prom}
\UnaryInfC{$!E \vdash {!}E$}
\AxiomC{}
\UnaryInfC{$E \vdash E$}
\RightLabel{\scriptsize$\multimap L$}
\BinaryInfC{$!E, \inta_A \vdash E$}
\RightLabel{\scriptsize cut}
\BinaryInfC{$!E \vdash E$}
\RightLabel{\scriptsize $\multimap R$}
\UnaryInfC{$\vdash \inta_A$}
\DisplayProof
\qquad
\tagarray{\label{cut_step_3}}
\end{center}
The next cut-elimination step \cite[\S 3.8.2]{mellies} transforms this proof to
\begin{center}
\AxiomC{$\church{2}_A$}
\noLine\UnaryInfC{$\vdots$}
\def\extraVskip{5pt}
\noLine\UnaryInfC{$!E \vdash E$}
\def\extraVskip{2pt}
\RightLabel{\scriptsize prom}
\UnaryInfC{$!E \vdash {!}E$}
\AxiomC{$\church{2}_A$}
\noLine\UnaryInfC{$\vdots$}
\def\extraVskip{5pt}
\noLine\UnaryInfC{$!E \vdash E$}
\def\extraVskip{2pt}
\RightLabel{\scriptsize cut}
\BinaryInfC{$!E \vdash E$}
\AxiomC{}
\UnaryInfC{$E \vdash E$}
\RightLabel{\scriptsize cut}
\BinaryInfC{$!E \vdash E$}
\RightLabel{\scriptsize $\multimap R$}
\UnaryInfC{$\vdash \inta_A$}
\DisplayProof
\quad
\tagarray{\label{cut_step_4}}
\end{center}
As may be expected, cutting against an axiom rule does nothing, so this is equivalent to
\begin{prooftree}
\AxiomC{$\church{2}_A$}
\noLine\UnaryInfC{$\vdots$}
\def\extraVskip{5pt}
\noLine\UnaryInfC{$!E \vdash E$}
\def\extraVskip{2pt}
\RightLabel{\scriptsize prom}
\UnaryInfC{$!E \vdash {!}E$}
\AxiomC{$\church{2}_A'$}
\noLine\UnaryInfC{$\vdots$}
\def\extraVskip{5pt}
\noLine\UnaryInfC{$!E, !E \vdash E$}
\def\extraVskip{2pt}
\RightLabel{\scriptsize ctr}
\UnaryInfC{$!E \vdash E$}
\RightLabel{\scriptsize cut}
\BinaryInfC{$!E \vdash E$}
\RightLabel{\scriptsize $\multimap R$}
\UnaryInfC{$\vdash \inta_A$}
\end{prooftree}
where $\church{2}_A'$ is a sub-proof of $\church{2}_A$. Here is the important step: cut-elimination replaces a cut of a promotion against a contraction by a pair of promotions \cite[\S 3.9.3]{mellies}. This step corresponds to the doubling of the promotion box in \eqref{diagram_mult_2_2}
\begin{prooftree}
\AxiomC{$\church{2}_A$}
\noLine\UnaryInfC{$\vdots$}
\def\extraVskip{5pt}
\noLine\UnaryInfC{$!E \vdash E$}
\def\extraVskip{2pt}
\UnaryInfC{$!E \vdash {!}E$}
\AxiomC{$\church{2}_A$}
\noLine\UnaryInfC{$\vdots$}
\def\extraVskip{5pt}
\noLine\UnaryInfC{$!E \vdash E$}
\def\extraVskip{2pt}
\UnaryInfC{$!E \vdash {!}E$}
\AxiomC{$\church{2}_A'$}
\noLine\UnaryInfC{$\vdots$}
\def\extraVskip{5pt}
\noLine\UnaryInfC{$!E, !E \vdash E$}
\def\extraVskip{2pt}
\RightLabel{\scriptsize cut}
\BinaryInfC{$!E, !E \vdash E$}
\RightLabel{\scriptsize cut}
\BinaryInfC{$!E, !E \vdash E$}
\RightLabel{\scriptsize ctr}
\UnaryInfC{$!E \vdash E$}
\RightLabel{\scriptsize $\multimap R$}
\UnaryInfC{$\vdash \inta_A$}
\end{prooftree}
We only sketch the rest of the cut-elimination process: next, the derelictions in $\church{2}_A'$ will be annihilate with the promotions in the two copies of $\church{2}_A$ according to \cite[\S 3.9.1]{mellies}. Then there are numerous eliminations involving the right and left $\multimap$ introduction rules. 

\section{Tangents and proofs}\label{section:example_lifting}




\begin{example}\label{example:tangent_coalgebra} Let $\cat{T}$ the coalgebra given by the dual of the finite-dimensional algebra $k[t]/(t^2)$. It has a $k$-basis $1 = 1^*$ and $\varepsilon = t^*$ and coproduct $\Delta$ and counit $u$ defined by
\[
\Delta(1) = 1 \otimes 1, \quad \Delta( \varepsilon ) = 1 \otimes \varepsilon + \varepsilon \otimes 1, \quad u(1) = 1, \quad u(\varepsilon) = 0\,.
\]
Recall that a tangent vector at a point $x$ on a scheme $X$ is a morphism $\Spec(k[t]/t^2) \lto X$ sending the closed point to $x$. Given a finite-dimensional vector space $V$ and $R = \Sym(V^*)$ with $X = \Spec(R)$, this is equivalent to a morphism of $k$-algebras
\[
\varphi: \Sym(V^*) \lto k[t]/t^2
\]
with $\varphi^{-1}( (t) ) = x$. Such a morphism of algebras is determined by its restriction to $V^*$, which as a linear map $\varphi|_{V^*}: V^* \lto k[t]/t^2$ corresponds to a pair of elements $(P, Q)$ of $V$, where $\varphi( \tau ) = \tau(P) \cdot 1 + \tau(Q) \cdot t$. Then $\varphi$ sends a polynomial $f$ to
\[
\varphi(f) = f(P) \cdot 1 + \partial_Q( f )|_P \cdot t\,.
\]
The map $\varphi|_{V^*}$ is also determined by its dual, which is a linear map $\phi: \cat{T} \lto V$. By the universal property, this lifts to a morphism of coalgebras $\Phi: \cat{T} \lto {!}V$. If $\phi$ is determined by a pair of points $(P,Q) \in V^{\oplus 2}$ as above, then it may checked directly that
\[
\Phi( 1 ) = \vacu_P, \qquad \Phi( \varepsilon ) = \ket{ Q }_P
\]
is a morphism of coalgebras lifting $\phi$.
\end{example}

Motivated by this example, we make a preliminary investigation into tangent vectors at proof denotations. Let $A,B$ be types with finite-dimensional denotations $\den{A}, \den{B}$.

\begin{definition} Given a proof $\pi$ of $\vdash A$ a \emph{tangent vector} at $\pi$ is a morphism of coalgebras $\theta: \cat{T} \lto {!} \den{A}$ with the property that $\theta(1) = \vacu_{\den{\pi}}$, or equivalently that the diagram
\begin{equation}
\xymatrix@C+2pc{
k \ar[d]_-{1} \ar[r]^-{\den{\pi}} & \den{A}\\
\cat{T} \ar[r]_-{\theta} & {!} \den{A} \ar[u]_-{d}
}
\end{equation}
commutes. The set of tangent vectors at $\pi$ is denoted $T_{\pi}$.
\end{definition}

It follows from Example \ref{example:tangent_coalgebra} that there is a bijection
\[
\den{A} \lto T_{\pi}
\]
sending $Q \in \den{A}$ to the coalgebra morphism $\theta$ with $\theta(1) = \vacu_{\den{\pi}}$ and $\theta(\varepsilon) = \ket{Q}_{\den{\pi}}$. We use this bijection to equip $T_{\pi}$ with the structure of a vector space.

Note that the denotation of a program not only maps inputs to outputs (if we identify inputs and outputs with vacuum vectors) but also tangent vectors to tangent vectors. To wit, if $\rho$ is a proof of a sequent $!A \vdash B$ with denotation $\lambda: {!} \den{A} \lto \den{B}$, then composing a tangent vector $\theta$ at a proof $\pi$ of $\vdash A$ with the lifting $\Lambda$ of $\lambda$ leads to a tangent vector at the cut of $\rho$ against the promotion of $\pi$. That is, the linear map
\begin{equation}\label{eq:fake_tangent_map}
\xymatrix@C+2pc{
\cat{T} \ar[r]^-{\theta} & ! \den{A} \ar[r]^-{\Lambda} & ! \den{B}
}
\end{equation}
is a tangent vector at the following proof, which we denote $\rho \l \pi$
\begin{prooftree}
\AxiomC{$\pi$}
\noLine\UnaryInfC{$\vdots$}
\def\extraVskip{5pt}
\noLine\UnaryInfC{$\vdash A$}
\def\extraVskip{2pt}
\RightLabel{\scriptsize prom}
\UnaryInfC{$\vdash {!} A$}
\AxiomC{$\rho$}
\noLine\UnaryInfC{$\vdots$}
\def\extraVskip{5pt}
\noLine\UnaryInfC{$!A \vdash B$}
\def\extraVskip{2pt}
\RightLabel{\scriptsize cut}
\BinaryInfC{$\vdash B$}
\end{prooftree}
By Theorem \ref{theorem:describe_lifting} the linear map of tangent spaces induced in this way by $\rho$ is
\begin{gather}
\den{A} \cong T_{\pi} \lto T_{\rho\l\pi} \cong \den{B}\label{eq:tangent_map_logic}\\
Q \longmapsto \lambda \ket{Q}_{\den{\pi}}\nonumber
\end{gather}

When $\rho$ computes a smooth map of differentiable manifolds, this map can be compared with an actual map of tangent spaces. We examine $\rho = \church{2}_A$ below. It would be interesting to understand these maps in more complicated examples; this seems to be related to the differential $\lambda$-calculus \cite{ehrhard_difflambda,ehrhard_difflambda2}, but we have not tried to work out the precise connection.

\begin{example}\label{example:tangent_to_2} When $k = \mathbb{C}$ and $Z = \den{\church{2}_A}_{nl}$ we have by Lemma \ref{lemma:nonlinear_recover}
\[
Z: M_n(\mathbb{C}) \lto M_n(\mathbb{C}), \qquad Z(\alpha) = \alpha^2\,.
\]
The tangent map of the smooth map of manifolds $Z$ at $\alpha \in M_n(\mathbb{C})$ is $(Z_*)_\alpha( \nu ) = \{ \nu, \alpha \}$. When $\alpha$ is the denotation of some proof $\pi$ of $\vdash A \multimap A$ this agrees with the tangent map assigned in \eqref{eq:tangent_map_logic} to the proof $\church{2}_A$ at $\pi$, using \eqref{eq:church_2_den}.
\end{example}

\bibliographystyle{amsalpha}
\providecommand{\bysame}{\leavevmode\hbox to3em{\hrulefill}\thinspace}
\providecommand{\href}[2]{#2}

\end{document}